\documentclass[11pt,reqno]{amsart}

\usepackage[T1]{fontenc}
\usepackage{amsmath}							
\usepackage{amssymb}
\usepackage{amsthm}
\usepackage{amscd}
\usepackage{amsfonts}
\usepackage{mathabx}
\usepackage{stmaryrd}
\usepackage[all]{xy}

\usepackage{caption}
\usepackage{subcaption}
\usepackage{euler}

\usepackage{extarrows}

\usepackage[colorlinks, linktocpage, citecolor = purple, linkcolor = blue]{hyperref}
\usepackage[noabbrev,capitalize]{cleveref}
\usepackage{color}
\usepackage{faktor}

 \usepackage{microtype}

\usepackage{tikz}									 \usepackage{float} 
\usetikzlibrary{matrix}
\usetikzlibrary{patterns}
\usetikzlibrary{positioning}
\usetikzlibrary{decorations.pathmorphing}
\usetikzlibrary{cd}
\definecolor{lilla}{RGB}{199,85,255}
\usepackage{ytableau}
\usepackage{fullpage}
\usepackage[shortlabels]{enumitem}

\newtheorem{maintheorem}{Theorem}	
\newtheorem{maincorollary}[maintheorem]{Corollary}	

\newtheorem{theorem}{Theorem}[section]
\newtheorem*{theorem*}{Theorem}
\newtheorem{lemma}[theorem]{Lemma}
\newtheorem*{lemma*}{Lemma}
\newtheorem{proposition}[theorem]{Proposition}
\newtheorem{corollary}[theorem]{Corollary} 
\newtheorem*{corollary*}{Corollary}

\theoremstyle{definition}
								
\newtheorem{definition}[theorem]{Definition}
\newtheorem{example}[theorem]{Example}
\newtheorem*{example*}{Example}

\newtheorem{remark}[theorem]{Remark}
\newtheorem*{remark*}{Remark}
\newtheorem{question}[theorem]{Question}

\newcommand{\R}{\mathbb{R}}
\newcommand{\ZZ}{\mathbb{Z}}

\newcommand{\xdownarrow}[1]{%
  {\left\downarrow\vbox to #1{}\right.\kern-\nulldelimiterspace}
}

\newcommand{\tGa}{\widetilde{\Gamma}}
\newcommand{\tga}{\widetilde{\gamma}}

\newcommand{\Ga}{\Gamma}

\newcommand{\La}{\Lambda}

\newcommand{\Si}{\Sigma}

\newcommand{\tj}{\widetilde{\jmath}}

\DeclareMathOperator{\Hom}{Hom}
\DeclareMathOperator{\rk}{rk}
\DeclareMathOperator{\Sym}{Sym}

\newcommand{\calA}{\mathcal{A}}
\newcommand{\calB}{\mathcal{B}}
\newcommand{\calC}{\mathcal{C}}

\newcommand{\calM}{\mathcal{M}}

\newcommand{\calT}{\mathcal{T}}

\newcommand{\wt}[1]{\widetilde{#1}}

\DeclareMathOperator{\val}{val}

\DeclareMathOperator{\Prym}{Prym}

\let\ddiv\relax
\DeclareMathOperator{\ddiv}{div}

\DeclareMathOperator{\Jac}{Jac}

\title{The tropical Abel--Prym map}

  \author{Giusi Capobianco}
  \address{Department of Mathematics\\ University of Roma Tor Vergata \\ 00133 Rome, Italy}
 \email{\href{mailto:capobianco@axp.mat.uniroma2.it}{capobianco@axp.mat.uniroma2.it}}
  \author{Yoav Len}
   \address{Mathematical Institute, University of St Andrews, St Andrews KY16 9SS, UK}
 \email{\href{mailto:yoav.len@st-andrews.ac.uk}{yoav.len@st-andrews.ac.uk}}

\begin{document}

\begin{abstract} 
 We prove that, under mild assumptions, the tropical Abel--Prym map $\Psi\colon \tGa\to\Prym(\tGa/\Ga)$ associated with a free double cover $\pi\colon \tGa\to \Ga$ is  harmonic  of degree $2$ if and only if the source graph $\wt\Gamma$ is hyperelliptic. This is in accordance with the  already established  algebraic result. 
 In this case, the Abel--Prym graph $\Psi(\tGa)$  is   hyperelliptic   of genus $g_{\Ga}-1$ and its Jacobian is isomorphic, as a pptav, to the Prym variety of the cover. We further show that the Abel--Prym graph coincides with a connected component of the tropical bigonal construction. 
 En route, we count the number of distinct free double covers by hyperelliptic metric graphs.  
\end{abstract}
\maketitle

\setcounter{tocdepth}{1}
\tableofcontents


\section{Introduction}

The Abel--Prym map relates the geometry of a  double cover of algebraic curves or metric graphs  with their corresponding Prym variety. Given a double cover $\wt X\to X$, 
the map is given by 

\[
 \begin{tabular}{cccc}	    $\Psi\colon$&$\Sym^d(\wt X)$&$\longrightarrow$&$\Prym(\wt X/X)$  \\
	         &$E$&$\mapsto$&$[E - \iota(E)]$ \\
	    \end{tabular},
\]
where $E$ is an effective divisor of degree $d$. It is the Prym-theoretic analogue of the Abel--Jacobi map.

The map plays a fundamental role in understanding the geometry of tropical Prym varieties. 
For instance, via a careful study of the $d$-fold Abel--Prym map  when $d=g-1$, 
the second author and Zakharov establish a weighted version of the Kirchhoff--Prym   formula and derive a geometric interpretation for the volume of the Prym variety of a free double cover of metric graphs \cite[Theorem 3.4]{LZ22}. 
They show that the Abel--Prym map is a \emph{harmonic morphism of polyhedral spaces} of global degree $2^{g-1}$ and determine the local degree at every cell of the decomposition, leading to a natural finite set of representatives  for elements of the Prym variety (cf. the case of Jacobians, where there is a unique choice of representative). 
As a counterpart, the appendix by Sebastian Casalaina-Martin 
shows that the degree of the algebraic Abel--Prym map is $2^{g-1}$ as well.
Motivated by this similarity, the second author and Zakharov
conjectured that the algebraic and the tropical Abel--Prym maps have the same degree for $d\leq g-2$ \cite[Conjecture 1.1]{LZ22}. 

In this paper we focus on the case $d=1$ for free double covers. 

For brevity, for the rest of this paper, we  refer to the 1-fold Abel--Prym map simply as the \emph{Abel--Prym map}. 
Our first result relates the degree of the Abel--Prym map with the structure of the graph $\wt\Gamma$.

\begin{maintheorem}\label{maintheorem:degreeOfAbelPrym}[\cref{thm:ifHyperelliptic}]
    Let $\pi\colon \tGa\to \Gamma$ be a free double cover of  metric graphs, where $\wt\Gamma$ is 2-edge-connected of genus at least 2.  Then $\wt\Gamma$ is hyperelliptic if and only if the Abel--Prym map  
    is a harmonic morphism of degree 2.
    \end{maintheorem}

Note that, even when  the Abel--Prym map is a  harmonic cover, it will often be ramified. \cref{maintheorem:degreeOfAbelPrym}
 is consistent with the analogous algebraic result \cite[Proposition 12.5.2(b)]{BL04}. 
However, we stress that tropically having degree 2 does not mean that the fibre over every point consists of two distinct points, since harmonic morphisms also allow for edge dilations and contractions. 
Furthermore, 
there are free double covers of non-hyperelliptic graphs in which the associated Abel--Prym map is not injective, see \cref{ex:hyperellipticType}. In fact, the Abel--Prym map turns out to not even be harmonic in general.

As exhibited in \cref{ex:bridgesRuinEverything}, the no-bridges requirement is necessary. This is not a phenomenon 
unique to Prym varieties: when a graph has bridges, the Abel--\emph{Jacobi} map is often not harmonic onto its image as well. The Abel--Prym map, however, may contract more than just bridges, see \cref{lem:contractedEdges}.
Somewhat surprisingly, the latter type of collapsible subgraphs does not interfere with harmonicity.

It is insightful to note that hyperellipticity plays a role in the analogous case for Jacobians as well. Indeed,  it is a direct result of Clifford's theorem (\cite[Chapter III]{ACGHI}  and \cite{Coppens_Clifford,Len_Clifford}) that a curve is  hyperelliptic if and only if the $2$-fold Abel--Jacobi map $\Sym^2(C)\to\Jac(C)$ is not injective. As a consequence, one can use  tropical techniques to study the extension of the  
universal Abel--Jacobi map to the boundary of the moduli space of curves \cite{AbreuAndriaPacini_AbelMaps}. In a similar way, it would be interesting to apply the tropical Abel--Prym map to study the algebraic universal Abel--Prym map. 
In this sense, the distinction between the algebraic and tropical versions should be seen not as a bug, but rather a feature, as it can be used to detect and explain phenomena related to  degenerations of double covers to the boundary of the moduli space of admissible covers.

In the final two sections, we continue to investigate the structure of hyperelliptic double covers and their Abel--Prym map. We refine a common construction of double covers of graphs in order to characterise the ones that are hyperelliptic. 
As a consequence,  \cref{cor:numberOfHyperelliptic} shows that, among the $2^g$ free double covers $\tGa\to \Ga$ of a hyperelliptic metric graph $\Ga$ of genus $g$, the source graph $\wt\Gamma$ is only hyperelliptic in $g+1$ of them. 
This is consistent with the algebraic result due to \cite{HMfarkashypcovers} that among the $2^{2g}$ étale double covers of a hyperelliptic smooth curve, only $\binom{2g+2}{2}$ are hyperelliptic. Indeed, the $2g+2$ fixed points tropicalise in pairs to the $g+1$ fixed points of the hyperelliptic involution on $\Ga$.
Since ${2g+2\choose 2}=(g+1)(2g+1)$, it seems plausible to speculate that the factor $(2g+1)$ is the lifting multiplicity of each double cover of  tropical hyperelliptic graphs arising from some Hurwitz count. But we leave such questions for future work.

The construction also naturally yields a basis for the space of antisymmetric cycles on the double cover, leading to a connection between the Jacobian of the Abel--Prym graph (namely the image of $\wt\Gamma$ by the Abel--Prym map) and the Prym variety of  the double cover, which is  a tropical analogue of \cite[Corollary 12.5.7]{BL04}.

\begin{maincorollary}[\cref{structureOfImage} and \cref{pptavs}]\label{mainCorollary:JacobianOfImage}
Let  $\pi\colon \tGa\to \Ga$ be a free  double cover of hyperelliptic metric graphs. Then the  Abel--Prym graph $\Psi(\wt\Gamma)$ is hyperelliptic  of genus $g_{\Ga}-1$ and its Jacobian
is isomorphic, as a pptav, to
the Prym variety of the cover, namely
    \[
    \Prym(\tGa/\Ga)\cong \mathrm{Jac}(\Psi(\tGa)).
    \]
    \end{maincorollary}

\noindent Furthermore, the Abel--Prym graph  is obtained by contraction or deletion of edges of $\wt\Gamma$, see \cref{structureOfImage} for a full description.

The corollary can be interpreted in the context of tropical moduli spaces as follows. Consider the  diagram
\[
        \begin{tikzcd}
\operatorname{Hyp}_g \arrow[d, "\Psi"] \arrow[r, "\Prym"] & \calA_{g-1}\\
\calM_{g-1}\arrow[ur, "\Jac"]                        &                  
\end{tikzcd},
\]
where $\operatorname{Hyp}_g$ is the hyperelliptic locus within the moduli space of tropical admissible double covers, $\calM_{g-1}$ is the moduli space of tropical curves of genus $g-1$, and $\calA_{g-1}$ is the moduli space of tropical abelian varieties of dimension $g-1$. 
The maps $\Prym$ and $\Jac$ associate a Prym variety and a Jacobian to a double cover and to a single tropical curve respectively. Then \cref{mainCorollary:JacobianOfImage} implies that the diagram commutes.   Notably, a construction that  assigns a graph to a double cover already exists in the literature and is known as the bigonal construction. Our final result asserts that these two constructions coincide in the hyperelliptic case.

\begin{maintheorem}[\cref{thm:bigonalconstruction} and \cref{cor:bigonalConstruction}]\label{mainTheorem:bigonal}
      Let $\pi\colon\tGa\to\Ga$ be a free double cover, where $\tGa$ and $\Ga$ are hyperelliptic metric graphs with involutions  $\tj$ and $j$ respectively,  and a finite Abel--Prym map. Then the output of the tropical bigonal construction consists of two connected components, one isomorphic to the Abel--Prym graph and the other a tree given as the quotient of $\wt\Gamma$ by its hyperelliptic involution. 
          Furthermore, the Prym variety of the bigonal construction is isomorphic to the Prym variety of the original double cover $\wt\Gamma/\Gamma$.
\end{maintheorem}

See also \cref{rem:bigonalForNonFiniteAP} and \cref{ex:bigonalForNonHyperelliptic} for the case where the Abel--Prym map is not finite or that the source graph is not hyperelliptic. 
\noindent In particular, \cref{mainTheorem:bigonal} generalises 
\cite[Theorem 5.6]{RZ_ngonal} to the case of  free double covers.

\medskip 

When studying the Abel--Prym map, one can also enquire for its class in the tropical homology of the Prym variety. As shown in \cite[Theorem 4.25]{RZ_ngonal}, this class has an explicit description that doesn't depend on the double cover. We find this phenomenon interesting, considering that properties of the double cover do play a significant role in the results of the current paper. However, we do not investigate the connection further and leave it for future work.

\subsection*{Open problems and future directions}

The results presented in this paper 
only apply to free double covers. However, we believe that they extend to any harmonic double cover (possibly with dilated edges and weighted vertices) with mild modifications. It would 
be  interesting to explore the Abel--Prym map in that case, as well as its connection with the bigonal construction. 
This would be a crucial step towards using the tropical Abel--Prym map to study the algebraic universal 1-fold Abel--Prym map.

Another natural direction is an extension to the $d$-fold Abel--Prym map for  $1 < d < g-1$. We believe that, in that case, the image of the Abel--Prym map is a balanced polyhedral complex within the Prym variety. When the graph $\wt\Gamma$ is hyperelliptic, an extension of our techniques should show
 that the Abel--Prym map is a harmonic map of degree $2^d$ onto its image. However, the behaviour of the $d$-fold Abel--Prym map beyond the hyperelliptic locus is expected to be much more complicated and to subtly depend  on the Brill--Noether locus of the graph.

\subsection*{Acknowledgements}
This project was initiated during G.C.'s visit at the University of St Andrews which was supported by the University of Rome Tor Vergata. G.C. would like to thank her supervisor Margarida Melo for the fruitful discussions during the period abroad. Y.L. was supported by the EPSRC New Investigator Award (grant number EP/X002004/1). We warmly thank Felix Röhrle for insightful comments on a previous version of this manuscript. We also thank Thibault Poiret and Martin Ulirsch for helpful discussions. 
A special thank goes to the anonymous referee for many  insightful comments and suggestions.

\section{Preliminaries}\label{sec:preliminaries}

We now recall some basic notions from the theory of metric graphs and their divisor theory and refer to \cite{BN09}, \cite{BF11}, \cite{BC}, and \cite{Len_Survey}  for a more complete introduction to the topic. 	
Let $G$ be a connected graph with $V(G)$ and $E(G)$  respectively the sets of vertices and edges. The \emph{genus} of $G$ is its  first Betti number $g=|E(G)|-|V(G)|+1$. 
A \emph{metric graph} of genus $g$ is a metric space $\Ga$ such that there exists a graph $G$ of genus $g$ and a length function \[l\colon E(G)\to \R_{>0}\]
so that $\Ga$ is obtained from $(G,l)$ by gluing intervals $[0,l(e)]$ for  $e\in E(G)$ at their endpoints, as prescribed by the combinatorial data of $G$. In that case, we say that $(G,l)$ is a \emph{model} for $\Ga$.  The model is called \emph{loopless} if $G$ has no loops. 
  For any point $x$, its \emph{valency}, denoted  $\val(x)$, is  the number of connected components in $U_x\setminus\{x\}$, where $U_x$ is an arbitrarily small open set containing $x$.   
  The \emph{canonical loopless model} $(G_0,l)$ for $\Ga$ is the model obtained by taking
\[
V=\{x\in\Ga|\mathrm{val}(x)\ne 2 \}
\]
and placing an additional vertex at the midpoint of each loop edge. We will often omit the finite graph and work directly with the metric graph. Unless otherwise mentioned, we will implicitly assume that we are working with the canonical  loopless model of the graph.

A metric graph is called \emph{2-connected} if it remains connected after removing a single point. It is called \emph{2-edge-connected} if it remains connected after the removal of an edge. In other words, a graph is 2-connected if it has no cut-vertices and it is  2-edge-connected if it has no bridges. Note that a 2-connected graph is 2-edge-connected but not vice versa. Furthermore, every graph becomes 2-edge-connected after contracting its bridges. 
 
 A \emph{divisor} $D$ is given by \[D=\sum_{p\in \Gamma}a_pp\] with $a_p\in\ZZ$ and $a_p=0$ for all but finitely many $p\in \Ga$ and 
it is called \emph{effective} if $a_p\geq 0$ for every $p\in\Ga$.
The \emph{degree} of $D $ is $\deg(D)=\sum_{p\in\Gamma}a_p$. 
We will denote by $ \mathrm{Div}^k_+(\Gamma)$ the set of effective divisors of degree $k$.
 
	A \emph{rational function} on $\Gamma$ is a continuous, piecewise linear function $f\colon \Gamma\to \R$  whose slopes are all integers. 
 
	The \emph{principal divisor} associated to the rational function $f\in \mathrm{PL}(\Gamma)$ is defined to be 
	\[
	\ddiv(f)=\sum_{P\in \Gamma}\sigma_p(f)p,
	\]
	 where $\sigma_p(f)$ is the sum of the slopes of $\Gamma$ in all directions emanating from $p$.

The \emph{complete linear system} associated to a divisor $D$ is
\[
|D|:=\{E\in\mathrm{Div}_+(\Ga) \,:E-D=\mathrm{div}(f), f\in PL(\Ga)\}.
\]

	  The \emph{rank} of a divisor $D$ on a metric graph $\Gamma$ 
	  is $r_{\Gamma}(D)=-1$ 
	   if $|D|=\emptyset$
	   , 
	   otherwise
	  \begin{equation*}
	  \begin{aligned}
	  &r_{\Gamma}(D)=\max\{k\in \ZZ: |D-E|\ne\emptyset,\,\forall\,E\in\mathrm{Div}^k_+(\Gamma) \}. 
	  \end{aligned}
	  \end{equation*}

\subsection{Morphisms of metric graphs}
We recall the definitions of hyperelliptic metric graphs and morphisms and refer to  \cite{CH, LenUlirschZakharov_abelianCovers} for details. 
Suppose that $(G,l)$ and $(G',l')$ are loopless models for metric graphs $\Ga$ and $\Ga'$. A \emph{morphism of loopless models} $\phi\colon(G,l)\to (G',l')$ is a map of sets 
\[
V(G)\cup E(G)\to V(G')\cup E(G')
\]
such that
\begin{itemize}
    \item[i)] $\phi(V(G))\subseteq V(G')$;
    \item[ii)] if $e=xy$ is an edge of $G$ and $\phi(e)\in V(G')$ then $\phi(x)=\phi(e)=\phi(y)$;
    \item[iii)] if $e=xy$ is an edge of $G$ and $\phi(e)\in E(G')$  then $\phi(e)$ is an edge between $\phi(x)$ and $\phi(y)$;
    \item[iv)] if $\phi(e)=e'$ then $l'(e')/l(e)$ is an integer.
\end{itemize}
Note that a morphism is allowed to contract edges, namely to send an edge of $G$ to a vertex of $G'$.

The morphism $\phi$ induces naturally a map $\tilde{\phi}\colon \Gamma\to \Gamma'$ of topological spaces. 
If $e\in E(G)$ is sent to $e'\in E(G')$, we declare $\tilde{\phi}$ to be linear along the edge $e$ and call $\mu_{\phi}(e)=l'(e')/l(e)$ the \emph{slope} of the map $\tilde\phi$.
The slope is an integer which is 0 on contracted edges.  We say that $\phi$ is \emph{finite} if $\mu_{\phi}(e)>0$ for all $e\in E(G)$. 
An edge $e\in E(G)$ mapping to an edge $e'$ is said to be \emph{dilated} if $\mu_{\phi}(e) >  1$.

\medskip

The following notion will be useful when describing properties of the Abel--Prym map later on.

\begin{definition}\label{def:injectivity}
 Let $f:\Gamma\to\Gamma'$ be a morphism of metric graphs.  We say that $f$ is \emph{globally injective} at a point $p$ if it is the only point of $\Gamma$ mapping to $f(p)$. It is globally injective at a subgraph $\Gamma_0$  if it is globally injective at every point of the subgraph (namely, if $p$ is in $\Gamma_0$ and $q$ is any point in $\Gamma\setminus\{p\}$, then $f(q)\neq f(p)$). 
\end{definition}

A morphism $\phi$ of loopless models is said to be \emph{harmonic} if for every $x\in V(G)$, the nonnegative integer 
\[
m_{\phi}(x)=\underset{x\in e, \,\phi(e)=e'}{\sum_{e\in E(G)}}\mu_{\phi}(e),
\]
is the same over all choices of $e'\in E(G')$ that are incident to the vertex $\phi(x)$.
The \emph{degree} of $\phi$ is 
\[
\deg\phi=\underset{\phi(e)=e'}{\sum_{e\in E(G)}}\mu_{\phi}(e)
\]
for any choice of $e'\in E(G')$.

A \emph{morphism} of metric graphs is a continuous   map $\widetilde\phi\colon \Ga\to \Ga'$ that is induced from a morphism $\phi\colon (G,l)\to (G',l')$ of loopless models. In that case, we define  $\deg\widetilde\phi:= \deg\phi$. The morphism $\wt\phi$ is \emph{harmonic} when $\phi$ is. An \emph{isomorphism} is a finite harmonic morphism of degree $1$. Note that every isomorphism is a bijection but not vice versa. For instance the map from a loop of length 1 to a loop of length 2 which dilates the edge by a factor of 2 is a bijective harmonic morphism of degree 2.  

A \emph{harmonic double cover} is a harmonic morphism of degree 2.
A double cover is called \emph{free} if the fibre over every point consists of exactly two distinct points. In particular, it doesn't have any dilated edges. A free double cover  induces a fixed-point free involution $\iota$ swapping the two points of the fibre over each point.

\medskip

A metric graph is \emph{hyperelliptic} if it has no leaf edges (namely vertices of valency 1) and has a divisor of degree $2$ and rank $1$. Hyperellipticity can be characterised by two other properties.

\begin{theorem*}\cite[Theorem 3.13]{CH}
Let $\Ga$ be a metric graph with canonical loopless model  $(G,l)$ that doesn't have leaf edges. Then the following are equivalent:
\begin{enumerate}
 \item $\Ga$ is hyperelliptic.
 \item There exists an involution $j\colon G\to G$ such that $G/j$ is a tree.
 \item There exists a finite harmonic morphism of degree $2$ from $\Gamma$ to a tree, or $|V(G)|=2$.
\end{enumerate}
\end{theorem*}

Here, the quotient is the graph obtained by identifying edges swapped by the involution and contracting edges swapped by the involution (note that this notion depends on the choice of model). From the no-leaves loopless assumption, the $|V(G)|=2$ case precisely captures the unique graph consisting of two vertices and edges between them, whose quotient is the trivial graph on a single vertex. 

\begin{remark}
    While our starting point in this paper are  double covers that are either free or hyperelliptic, the resulting Abel--Prym map will often have dilations, see \cref{lem:edgeMappedWithMultiplicity2}.
\end{remark}

In \cref{sec:hyperellipticDoubleCovers} we will count the number of hyperelliptic double covers of a given hyperelliptic graph. To that end, we require  an appropriate notion of an equivalence of such covers:
an \emph{isomorphism}  of double covers is an isomorphism $\phi\colon\Gamma_1\to \Gamma_2$ such that the diagram
\[
\begin{tikzcd}
\Gamma_1 \arrow[r, "\phi"] \arrow[d, "\pi_1"] & \Gamma_2 \arrow[d, "\pi_2"] \\
\Gamma \arrow[r, "Id"]                        & \Gamma                     
\end{tikzcd}
\]
commutes.

\medskip
The following result, concerning  the relationship between  hyperellipticity and free double covers,  will be used throughout the paper.

\begin{proposition}\label{commutative}
    Let $\pi\colon \tGa\to \Gamma$ be a free double cover of  metric graphs of genus $g_\Gamma\geq 2$. If $\wt\Gamma$ is hyperelliptic 
     then 
$\Gamma$ is hyperelliptic as well and the diagram
         \[
        \begin{tikzcd}
\tGa \arrow[d, "\pi"] \arrow[r, "\tj"] & \tGa\arrow[d, "\pi"] \\
\Ga\arrow[r, "j"]                        & \Ga                        
\end{tikzcd}
        \]

         commutes, where $j$ and $\tj$ are the hyperelliptic involutions of $\Gamma$  and $\wt\Gamma$ respectively.

    \begin{proof} 
     Since equivalent divisors map under harmonic morphisms to equivalent divisors \cite[Proposition 1.4.2]{LenUlirsch_Prym}, it follows that given a divisor  $D$ of degree $2$ and rank $1$ on  $\tGa$, the rank of $\pi_*(D)$ is at least 1. On the other hand, Clifford's theorem \cite[Theorem 1]{Facchini_Clifford} implies that the rank is exactly 1. It follows that $\Gamma$ is hyperelliptic with involution $j$, where $j(y)$ is the unique point such that $y+j(y)$ is the divisor of degree 2 and rank 1. Since, for every $x\in\wt\Gamma$, its involution $\tj(x)$ is the unique point such that $x+\tj(x)$ is of degree 2 and rank 1, and since equivalent divisors on $\wt\Gamma$ map to equivalent divisors on $\Gamma$, we conclude that $\tj$ commutes with $\pi$.

    We now turn to show that $\iota$ and $\tj$ commute, recall that $\iota$ is the fixed-point-free involution associated to $\pi$ while $\tj$ has fixed points. 
     If $p$ is a point where $\iota$ and $\tj$ coincide then, since both maps are involutions, we see that $\iota\tj(p) = p = \tj\iota(p)$. We may therefore assume that $\tj(p)\neq\iota(p)$.

       Suppose first that $\pi(p)$  is contained in a bridge of $\Gamma$. Then its preimages in $\wt\Gamma$ are either bridges or a pair of edges whose removal disconnects the graph. 
              We claim that, in the latter case, we have $\tj p = \iota p$, which is a contradiction: by the uniqueness of the hyperelliptic involution, it suffices to show that $p+\iota p$ has rank 1. As bridges are fixed by the hyperelliptic involution, it follows that the divisor $\pi(p + \iota p) = 2\pi(p)$ has rank 1.  
Let $x$ be any point of $\wt\Gamma$ and denote $y=\pi(x)$.
 
Let $f$ be the piecewise linear function on $\Gamma$ such that \[2\pi(p)+\ddiv f=y+y'\]  for some $y'$. 
The pullback  $\pi^*(f)$ of $f$ is a piecewise linear function on $\tGa$ such that $\ddiv(\pi^*(f))$ is supported on $x$ and $\iota x$. Since $x$ was arbitrary, we conclude that $p+\iota p$ has rank 1, so $\iota p = \tj p$, a contradiction. In particular, we can assume that $\tj p$ belongs to a bridge of $\wt\Gamma$. But then $\tj$ fixes both $p$ and $\iota p$ and, in particular, commutes with $\iota $.

        If, instead, $\pi(\widetilde{\jmath}p)$ is contained in a cycle, then by construction of the hyperelliptic involution, we have
        \begin{equation}\label{c}
             p+\widetilde{\jmath}p\simeq \iota p+\tj\iota p.
        \end{equation}  
        Using the fact that $\pi_*$ preserves linear equivalence and is invariant under $\iota$, we get
        $\pi(\widetilde{\jmath}p)\simeq \pi(\widetilde{\jmath}\iota p).
        $
        Since $\tj p$ and $\tj\iota(p)$
        are contained in a cycle, the linear equivalence is, in fact, an equality. Moreover, by definition of $\iota$ we also have 
        $\pi(\widetilde{\jmath}p)=\pi(\iota \tj p)$. Therefore,
        \[
\pi(\iota\widetilde{\jmath}p)=\pi(\widetilde{\jmath}p)=\pi(\widetilde{\jmath}\iota p).
        \]
But, as $\pi$ is a free \emph{double} cover,  either $\tj\iota p=\iota \tj p$ or $\tj\iota p=\tj p$. In the former case, we are done. In the latter case, since $\tj$ is an involution, it follows that $\iota p = p$, which is a contradiction since $\iota$ doesn't have fixed points.

    \end{proof}
\end{proposition}

\begin{example}
\cref{abel1} shows a hyperelliptic double cover of a theta graph with the hyperelliptic involutions $\tj, j$, as well as the double cover $\pi$. As we see, the two maps commute with each other. 

\begin{figure}[H]
    \centering
\begin{tikzpicture}[scale=0.8]

\coordinate (2) at (2,1.5); 
 \coordinate (a2) at (2,5); 
 \coordinate (b2) at (2,8); 
 
\coordinate (12) at (-1,1.5); 
 \coordinate (a12) at (-1,5); 
\coordinate (b12) at (-1,8); 
 \coordinate (b5) at (-0.4,8.7); 
\coordinate (a5) at (-0.4,4.3);
\coordinate (5) at (-0.4,2.2); 
 \coordinate (c1) at (0.4,6.4); 
 \coordinate (c2) at (0.6,6.6);

 \coordinate (2') at (8,1.5); 
 \coordinate (a2') at (8,5); 
 \coordinate (b2') at (8,8); 
 
\coordinate (12') at (5,1.5); 
 \coordinate (a12') at (5,5); 
\coordinate (b12') at (5,8); 
 \coordinate (b5') at (7.4,8.7); 

\coordinate (5') at (7.4,2.2); 
 \coordinate (c1') at (0.4+6,6.4); 
 \coordinate (c2') at (0.6+6,6.6);
 
 \foreach \i in {2,12,a2,a12,b2,b12,2',12',a2',a12',b2',b12'}
 \draw[fill=black](\i) circle (0.15em);

\draw[black,fill=black](b5) circle (0.20em);
\draw[black,fill=black](a5) circle (0.20em);
\draw[black,fill=black](5) circle (0.20em);
 \draw[thin, black] (a12)--(a2);
 \draw[thin, black] (b12)--(b2);
 \draw[thin, black] (12) to [out=90, in=90] (2);
\draw[thin, black] (12) to [out=90+180, in=-90] (2);
\draw[thin, black] (12)--(2);
\draw[thin, black] (b12) to [out=90, in=90] (b2);
\draw[thin, black] (a12) to [out=90+180, in=-90] (a2);

  \draw[thin, black] (b12)--(a2);
\draw[thin, black] (a12)--(c1);
    \draw[thin, black] (b2)--(c2);
  \draw (0.5,3.5) node [below] {$\downarrow$};
   \draw (0.8,3.4) node [below] {$\pi$};

   \draw (b5) node [above] {$p$}; 
  \draw (-0.4,2.2) node [above] {$\pi(p)$}; 
    \draw (3,6.5) node [right] {$\rightarrow$};
     \draw (3.3,6.5) node [above] {$\tj$};
      \draw (3,1.5) node [right] {$\rightarrow$};
      \draw (3.3,1.5) node [above] {$j$};
   \draw (a5) node [below] {$\iota(p)$};

\draw[black,fill=black](b5') circle (0.20em);
\draw[black,fill=black](5') circle (0.20em);
 \draw[thin, black] (a12')--(a2');
 \draw[thin, black] (b12')--(b2');
 \draw[thin, black] (12') to [out=90, in=90] (2');
\draw[thin, black] (12') to [out=90+180, in=-90] (2');
\draw[thin, black] (12')--(2');
\draw[thin, black] (b12') to [out=90, in=90] (b2');
\draw[thin, black] (a12') to [out=90+180, in=-90] (a2');

  \draw[thin, black] (b12')--(a2');
\draw[thin, black] (a12')--(c1');
    \draw[thin, black] (b2')--(c2');
  \draw (0.5+6,3.5) node [below] {$\downarrow$};
   \draw (6.8,3.4) node [below] {$\pi$};

   \draw (b5') node [above] {$\tj(p)$}; 
  \draw (8.8,2.25) node [above] {$\pi(\tj(p))=j(\pi(p))$};

  \end{tikzpicture}
 \caption{The hyperelliptic involution and the double cover maps for the theta graph.}
    \label{abel1}
\end{figure}

\end{example}

\subsection{The Prym variety and the Abel--Prym map}

Following \cite[Section 1.4]{LenUlirsch_Prym}, the map $\pi$ induces a \emph{norm} map $\pi_*:\Jac(\wt\Gamma)\to\Jac(\Gamma)$, sending a divisor class $[\sum p_i - \sum q_i]$ to $[\sum \pi(p_i) - \sum \pi(q_i)]$. The kernel of the map has two connected components, known as the \emph{even} and \emph{odd}
 Prym varieties, where the former is the component containing the identity.
 Throughout this paper, we will only care about the odd component and simply refer to it as the \emph{Prym variety}.    The 1-fold Abel--Prym map is then given by
 \[
 \begin{tabular}{cccc}	    $\Psi\colon$ & $\wt \Gamma $ & $\longrightarrow$&$\Prym(\wt \Gamma/\Gamma)$  \\
	         &$p$&$\mapsto$&$[p - \iota(p)]$ \\
	    \end{tabular}.
\]
We will from now on omit the term 1-fold and refer to the map simply as the \emph{Abel--Prym map}.

\begin{definition}
    The image of the  Abel--Prym map is called the \emph{Abel--Prym graph}.

\end{definition}

In \cref{sec:JacobianVSPrym} we  provide a detailed description of the metric on the Prym variety. However, since most of the details are not necessary before that point, we now only describe the metric  on the Abel--Prym graph. Fix a basis $B$ for antisymmetric cycles on $\wt\Gamma$ (namely, such that each oriented cycle  $\gamma$ in $B$ satisfies $\iota\gamma = -\iota\gamma$). Using  \cite[Construction B]{LZ22}, the basis may be chosen so that the multiplicity on each edge of every element  (with respect to a fixed orientation) is either $0,\pm 1,$ or  $\pm 2$.  Let $\tilde e$ be a segment of $\wt\Gamma$ of length $\ell(\tilde e)$, short enough so that its  intersection with every element of $B$ has a fixed multiplicity. Then 
\begin{equation}\label{eq:PrymMetric}
\Psi(\tilde e) = 
\begin{cases}
 \text{a point}, & \text{if $\tilde e$ has trivial intersection with all elements of $B$;} \\

 \text{a segment of length $2\ell(\tilde e)$}, & \text{if $\tilde e$ intersects all the elements of $B$ with even multiplicity;}\\

 \text{a segment isometric to $\tilde e$}, & \text{otherwise.}
 
\end{cases}
\end{equation}

\section{Basic properties  of the 1-fold Abel--Prym map}
\label{sec:behaviourofAP}

Fix a free double cover $\pi:\wt\Gamma\to\Gamma$ throughout. 
In this section, we determine the  multiplicity of the Abel--Prym map on various types of edges of the double cover and use it to prove \cref{maintheorem:degreeOfAbelPrym}. Most of the structural results concerning the Abel--Prym map are summarised in 
\cref{thm:behaviousOnEdges}, whose proof will be given in a series of lemmas throughout the section.

\begin{theorem}\label{thm:behaviousOnEdges}
  Let $\pi\colon\wt\Gamma\to\Gamma$ be a free double cover. Let $e$ be an edge of $\Gamma$ with preimages $e',e''$ in $\wt\Gamma$.

  \begin{enumerate}
      \item If $e'$ and $e''$ are bridges then they are contracted by $\Psi$ (note that, if one of them is a bridge then the other one is as well).

            \item If $e$ is a bridge and $e',e''$ are a properly disconnecting pair then              
             $\Psi$ is globally injective 
             on $e',e''$  
            and dilates both.

           \item If $e$ is not a bridge and $e',e''$ form  a properly disconnecting pair  then they are contracted by $\Psi$.

\item If $e$ is not a bridge and the removal of $e'$ and $e''$ does not disconnect $\wt\Gamma$,  then $\Psi$ maps $e'$ and $e''$ isometrically and  $\Psi^{-1}(\Psi(e'))$ consists of either 1 or 2 edges of $\wt\Gamma$. 
      
  \end{enumerate}
\end{theorem}

\noindent See Definitions \ref{def:injectivity}  and  \ref{def:propdiscpair} for an explanation of ``global injectivity" and ``properly disconnecting pairs" respectively.

\begin{proof}

Parts (1) and (3) follow from \cref{lem:contractedEdges} below.  Part (3) follows from \cref{lem:edgeMappedWithMultiplicity2}.
Part (4) follows \cref{lem:notBridgeNotDisconnectingPair} together with \cref{prop:fibreSize}.  
    
\end{proof}

A more refined description of the fibres of the Abel--Prym map is given in  \cref{lem:AbelPrymSeparatesTwoConnectedComponents} and \cref{lem:AbelPrymSeparatesTwoConnectedComponents2}
below. 
We begin the proof of the various parts of \cref{thm:behaviousOnEdges} by identifying the cases  in which the Abel--Prym map contracts an edge.

\begin{lemma}\label{lem:contractedEdges}
    Let $e$ be an edge of $\Gamma$ with preimages $e',e''$ in $\wt\Gamma$. The Abel--Prym map contracts $e'$ and $e''$ precisely in the following two cases. 
    \begin{enumerate}
        \item $e,e',e''$ are all bridges.
        
        \item $e$ is not a bridge and $\wt\Gamma\setminus\{e',e''\}$ is disconnected. 
    \end{enumerate}
\end{lemma}

\begin{proof}
According to \cite[Theorem 4.1]{LZ22}, the Abel--Prym map contracts the edge $e'$ if and only if  the number     of connected components of $\wt\Gamma\setminus\{e',e''\}$ is larger than the number of connected components of $\Gamma\setminus{e}$. But this happens precisely in the two cases above.
\end{proof}

See \cref{figpic9}.(B) and \cref{figpic9}.(C) for examples of the two types of contracted edges. 
Note that, if $e'$ is a bridge then so are $e$ and $e''$. Vice versa, when $e$ is not a bridge, it  always holds that neither  of $\{e',e''\}=\pi^{-1}(e)$ is a bridge. This situation will repeat often in the sequel so we give it a name.   

\begin{definition}\label{def:propdiscpair}
    A pair of edges $\{e',e''\}$ is said to form a  \emph{properly disconnecting pair} if removing both of them disconnects the graph but removing only one of them does not. 
\end{definition}

\begin{remark}
 \cite{RZ_matroidal} provides  a matroidal perspective on  tropical Prym varieties and the two cases described in \cref{lem:contractedEdges} coincide with the 1-circuits of the matroid ${M}^*(\wt G/G)$ (see \cite[Figure 3]{RZ_matroidal} for an example). Moreover, such edges are exactly those whose contraction does not change the Prym variety.
\end{remark}

The contracted locus plays a significant role in the results of this paper, so we give it a name as well. 
\begin{definition}
    We refer to the subset of $\wt\Gamma$ contracted by the Abel--Prym map  as the \emph{$\Psi$-collapsible locus}.  The subset arising from the second kind appearing in 
    \cref{lem:contractedEdges} is referred to as the \emph{cyclic} $\Psi$-collapsible locus. 
\end{definition}

The complement of the   cyclic $\Psi$-collapsible locus   within the $\Psi$-collapsible locus consists of the bridges of $\wt\Gamma$, hence the adjective `cyclic'.

\medskip 
We now study non-bridge edges of $\wt\Gamma$ that map to bridges of $\Gamma$.

\begin{lemma}\label{lem:edgeMappedWithMultiplicity2}
    Let $e$ be a bridge  of $\Gamma$ and suppose that its preimages  $e',e''$ form a properly disconnecting pair.  
    Then $\Psi$ dilates $e',e''$ by a factor of 2 and is globally injective on their interior.
    \end{lemma}

\begin{proof}

    Fix an orientation on $\Gamma$
 and a  consistent orientation on $\wt\Gamma$.    Using  \cite[Construction B]{LZ22}, there is a basis for the anti-symmetric cycles of $\wt\Gamma$, such that every cycle $\gamma$ that passes through $e'$ also passes through $e''$ with opposite orientation and satisfies $\gamma(e') = \pm 2$. It then follows from the definition of the metric on the image of the Abel--Prym map  (\cref{eq:PrymMetric} in \cref{sec:preliminaries}) that the image of $e'$ via $\Psi$ will have twice the length as $e'$.

We now turn to the  injectivity of $\Psi$ at $e'$. We will assume that $\wt\Gamma$ doesn't have any bridges, since they get contracted by the Abel--Prym map and don't play a role here. Suppose that $p$ is a point in the interior of $e'$ and that $q$ is another point of $\wt\Gamma$ such that  $\Psi(p) = \Psi(q)$, or what is the same,  $p+\iota q\simeq q+\iota p$. 

Denote $\phi$ the piecewise linear function such that $p+\iota q = q + \iota p + \ddiv\phi$. By moving $p+\iota q$ a short distance towards $q+\iota p$, we can assume that both $q$ and $\iota p$ are in the interior of edges.

Since   $e$ is a bridge, it follows that the two connected components of $\wt\Gamma\setminus\{p,\iota p\}$ are invariant under the involution $\iota$. In particular, $q$ and $\iota q$ belong to the same component of $\wt\Gamma\setminus\{p,\iota p\}$.
Denote $\Sigma_1,\Sigma_2$ the connected components of $\wt\Gamma\setminus\{p,\iota p\}$ such that $q,\iota q\in\Sigma_1$. We will show that under our assumptions, $\phi$ cannot be continuous.

Since $p + \iota q$  is a movable divisor (namely it is equivalent to at least one other effective divisor), the complement of $\{p,\iota q\}$  is disconnected. By the no-bridge assumption, both of the components of the complement are connected. 
We claim that $q$ and $\iota p$ are both on the same  connected component of the complement. 
Indeed, the function $\phi$ may only bend upwards at $p$ and $\iota q$ and only bend downwards to $q$ and $\iota p$. Therefore, if $A$ is the region where $\phi$  obtains its minimal value, its boundary consists of $p$ or $\iota q$ or both. If it consists of both then the claim is proven. 
If, on the other hand, the boundary consists of only one of the points, then $\phi$ may only be continuous if this point belongs to a bridge, which is a contradiction. 

Note that $A$ is fully contained in $\Sigma_1$.  
In particular, the function $\phi$ must bend upwards at a tangent direction from $p$ into $\Sigma_2$. Let $\gamma$ be a path from $p$ to $\iota p$ in $\Sigma_2$ that starts at this tangent direction. Since $\phi$ only bends downwards at $q, \iota p$, we may assume that $\phi$ has slope 1 along the entire path. 
By the assumption that $\Sigma_2$ is, itself, a connected double cover, it must have genus at least $1$, 
so there must be another path  $\gamma'$ in $\Sigma_2$ between $p$ and $\iota p$.
The  slope of $\phi$ along $\gamma'$ is 1 wherever $\gamma$ and $\gamma'$ intersect and 0 everywhere else. But since both paths terminate at $\iota p$, the function $\phi$ cannot be continuous, a contradiction.

\end{proof}

Next, we deal with the case where $e$ is not a bridge and its preimages in $\wt\Gamma$ are not a disconnecting pair. 

\begin{lemma}\label{lem:notBridgeNotDisconnectingPair}
      Let $e$ be an edge of $\Gamma$ that is not a bridge  and suppose that the removal of its preimages  $e',e''$ doesn't disconnect $\wt\Gamma$.  
    Then $\Psi$ maps $e'$ and $e''$ isometrically onto their image. 
\end{lemma}
 
\begin{proof}
Again, using    \cite[Construction B]{LZ22}, we have a basis for the anti-symmetric cycles of $\wt\Gamma$ such that at least one cycle passes through $e'$ with multiplicity $\pm 1$. The result now follows directly from the description of the metric on the image of the Abel--Prym map \cref{eq:PrymMetric}.
\end{proof}

We proceed to exhibit more refined properties of the Abel--Prym map. We begin by showing that,  under mild conditions, whenever the Abel--Prym map identifies two points of $\wt\Gamma$, their images in $\Gamma$ via $\pi$ belong to the same  2-connected component, namely, they are not separated by a bridge. In what follows, we say that $\Psi$ is \emph{finite} at a point $p$ if it is not in the interior of an edge contracted by $\Psi$.  

\begin{lemma}\label{lem:AbelPrymSeparatesTwoConnectedComponents}
Suppose that $\Gamma_1$ and $\Gamma_2$ are two connected graphs and that $\Gamma$ is the graph obtained by joining $\Gamma_1$ and $\Gamma_2$ along a point or a bridge, denoted $x$. 
Let $p$ and $q$ be distinct points of $\wt\Gamma$ where $\Psi$ is finite and $\Psi(p) = \Psi(q)$. Then $\pi(p)$ and $\pi(q)$ don't belong to the interior of $x$ and the following hold.
\begin{enumerate}
    \item If the preimages of $\Gamma_1$ and $\Gamma_2$ in $\wt\Gamma$ are connected then $\pi(p)$ and $\pi(q)$ belong to the closure of the same connected component of $\Gamma\setminus x$.

    \item If $\pi^{-1}(\Gamma_1)$ is connected but $\pi^{-1}(\Gamma_2)$ consists of two connected components, then either $p,q$ both belong to $\pi^{-1}(\Gamma_1)$ or they belong  to distinct components of $\pi^{-1}(\Gamma_2)$.
\end{enumerate}

\end{lemma}

\begin{proof}
First, since $p$ and $q$ are distinct and $\Psi$ is finite at those points, it follows from \cref{lem:contractedEdges} and \cref{lem:edgeMappedWithMultiplicity2}   that $\pi(p)$ and $\pi(q)$ do not belong to the interior of $x$.   
As $\Psi(p) = \Psi(q)$, let $f$ be the piecewise linear function such that 
    \[
    \ddiv(f) = (q - \iota q) - (p-\iota p). 
    \]
    Then $f$ may only bend upwards at $p$ and $\iota q$ and may only bend downwards at $q$ and $\iota p$. Let $A$ be the region of $\wt\Gamma$ where $f$ obtains its minimum. Then $\ddiv(f) < 0$ on the boundary of $A$, so $A$ is connected with boundary points $p$ and $\iota q$. 
    
    Now, assume that both $\pi^{-1}({\Gamma_1})$ and $\pi^{-1}(\Gamma_2)$ are connected. For the sake of contradiction, assume that  $\pi(p)$ and $\pi(q)$ don't belong to the same  $\Gamma_i$.
    In particular,  $A$  must contain at least one of the lifts of $x$, denoted $x'$ and $x''$. Since $\ddiv f$ is antisymmetric, its slopes are antisymmetric, namely, the slope of $f$ on any edge $\iota e$ equals minus its slope on $e$. In particular, $f$ must be constant on both $x'$ and $x''$.  It then follows that, if $f_1$ is the restriction of $f$ to $\pi^{-1}(\Gamma_1)$, then $\ddiv f$ and $\ddiv f_1$ coincide on $\pi^{-1}(\Gamma_1)$ (and, as always, they coincide in the interior). That is, 
     $p + \ddiv f_1 = \iota p$. This is only possible if  $p$ and $\iota p$ are connected by a single disconnecting path or $p = \iota p$, both of which are impossible on a free double cover. 

Now, assume that  $\pi^{-1}(\Gamma_2)$ is disconnected. Then it must be that its connected components are swapped by the involution and that $\pi^{-1}(\Gamma_1)$ is connected (otherwise $\wt\Gamma$ is not connected).  
If $\pi(q)\in\Gamma_1$ and $\pi(p)\notin\Gamma_1$ then, as before, the set  $A$ contains one of the preimages of $x$. In particular, we see that $\iota q$ is equivalent to $q$ on $\pi^{-1}(\Gamma)$, which leads to a contradiction. The case $\pi(q)\in\Gamma_2$ is similar. 

\end{proof}

\begin{example}
\cref{figpic9}.(A) and \cref{figpic9}.(B) depict  the two situations in 
\cref{lem:AbelPrymSeparatesTwoConnectedComponents}. In the first case the two points $p$ and $q$ belong to the same connected component while, if the preimage is not connected, the two points belong to different components. Moreover, if the edge $e$ is not a bridge and lifts to a disconnecting pair as in (C), then $p$ and $q$ have to belong to distinct components in order to have $p+\iota q\simeq q+\iota p$.
This is the content of the next lemma, \cref{lem:AbelPrymSeparatesTwoConnectedComponents2}.
      \begin{figure}[h]
\begin{subfigure}[t]{0.4\linewidth}
\centering
\begin{tikzpicture}
\coordinate (2) at (0.5,0);
\coordinate (3) at (2,0);
\coordinate (a2) at (0.5,1.5);
\coordinate (a3) at (2,1.5);
\coordinate (b2) at (0.5,4);
\coordinate (b3) at (2,4);
  \coordinate (p) at (0.2,2.8);
  \coordinate (q) at (0.5,2.8);
  
 \foreach \i in {2,3,a2,b2,a3,b3}
 \draw[fill=black](\i) circle (0.15em);
 \draw[red,fill=red](p) circle (0.15em);
  \draw[red,fill=red](q) circle (0.15em);

\draw[thin, black] (2)--(3);
\draw[thin, black] (a2)--(a3);
\draw[thin, black] (b2)--(b3);

\draw[thin, black] (2) to [out=90, in=40] (0.3,0.3) to [out=40+180, in=70] (-0.1,0.2) to [out=70+180, in=65] (-0.3,0) ;
\draw[thin, black]  (-0.3,0)  to [out=65+180, in=140]  (-0.15,-0.25)  to [out=140+180, in=140] (0.3,-0.2) to [out=140+180, in=-90](2) ;

 \draw[thin, black] (3) to [out=90, in=120] (2.3,0.3) to [out=120+180, in=70+45] (2.65,0.2) to [out=70+30+180, in=100] (2.8,0) ;
\draw[thin, black]  (2.8,0)  to [out=270, in=140]  (2.5,-0.25)  to [out=140+180, in=90-180] (2.2,-0.2) to [out=90, in=-90](3) ;

\draw[thin, black] (b2) to [out=180+45, in=110] (0.2,3.6) to [out=110+180, in=40] (0,3) to 
[out=40+180, in=80] (0.1,2) to
[out=80+180, in=90] (a2) ;
\draw[thin, black]  (b2)  to [out=-65, in=140]  (0.8,3.5)  to [out=140+180, in=140] (0.8,2.3) to [out=140+180, in=45](a2) ;

\draw[thin, black] (b3) to [out=-65, in=110] (1.9,3.6) to [out=110+180, in=45] (1.8,2.8) to 
[out=45+180, in=80] (2.1,2) to
[out=80+180, in=90] (a3) ;
\draw[thin, black]  (b3)  to [out=-40, in=140]  (2.5,3.5)  to [out=140+180, in=140] (2.5,2.5) to [out=140+180, in=45](a3) ;

  \draw (1.25,1) node [below] {$\downarrow$};
\draw (1.25,0.4) node [below] {$e$};
  
  \draw (1.25,2) node [below] {$e'$};
  
  \draw (1.25,4.5) node [below] {$e''$};
  
  \draw[red] (p) node [below] {$p$};
  
  \draw[red] (q) node [below] {$q$};
  \end{tikzpicture}
 \caption{}
   \end{subfigure}
\hfill
\begin{subfigure}[t]{0.4\linewidth}
\centering
\begin{tikzpicture}
\coordinate (2) at (0.5,0);
\coordinate (3) at (2,0);
\coordinate (a2) at (0.5,1.5);
\coordinate (a3) at (2,1.5);
\coordinate (b2) at (0.5,4);
\coordinate (b3) at (2,4);
      
 \coordinate (p) at (2.2,4);
  \coordinate (q) at (2.5,1.5);
  
 \foreach \i in {2,3,a2,b2,a3,b3}
 \draw[fill=black](\i) circle (0.15em);

\draw[red,fill=red](p) circle (0.15em);
  \draw[red,fill=red](q) circle (0.15em);
  
\draw[thin, black] (2)--(3);
\draw[thin, black] (a2)--(a3);
\draw[thin, black] (b2)--(b3);

\draw[thin, black] (2) to [out=90, in=40] (0.3,0.3) to [out=40+180, in=70] (-0.1,0.2) to [out=70+180, in=65] (-0.3,0) ;
\draw[thin, black]  (-0.3,0)  to [out=65+180, in=140]  (-0.15,-0.25)  to [out=140+180, in=140] (0.3,-0.2) to [out=140+180, in=-90](2) ;


 \draw[thin, black] (3) to [out=90, in=120] (2.3,0.3) to [out=120+180, in=70+45] (2.65,0.2) to [out=70+30+180, in=100] (2.8,0) ;
\draw[thin, black]  (2.8,0)  to [out=270, in=140]  (2.5,-0.25)  to [out=140+180, in=90-180] (2.2,-0.2) to [out=90, in=-90](3) ;

\draw[thin, black] (b2) to [out=180+45, in=110] (0.2,3.6) to [out=110+180, in=40] (0,3) to 
[out=40+180, in=80] (0.1,2) to
[out=80+180, in=90] (a2) ;
\draw[thin, black]  (b2)  to [out=-65, in=140]  (0.8,3.5)  to [out=140+180, in=140] (0.8,2.3) to [out=140+180, in=45](a2) ;

 \draw[thin, black] (a3) to [out=90, in=120] (2.3,0.3+1.5) to [out=120+180, in=70+45] (2.65,0.2+1.5) to [out=70+30+180, in=100] (2.8,1.5) ;
\draw[thin, black]  (2.8,1.5)  to [out=270, in=140]  (2.5,-0.25+1.5)  to [out=140+180, in=90-180] (2.2,-0.2+1.5) to [out=90, in=-90](a3) ;

 \draw[thin, black] (b3) to [out=90, in=120] (2.3,0.3+4) to [out=120+180, in=70+45] (2.65,0.2+4) to [out=70+30+180, in=100] (2.8,+4) ;
\draw[thin, black]  (2.8,4)  to [out=270, in=140]  (2.5,-0.25+4)  to [out=140+180, in=90-180] (2.2,-0.2+4) to [out=90, in=-90](b3) ;
  \draw (1.25,1) node [below] {$\downarrow$};
 
  \draw (1.25,0.4) node [below] {$e$};
  
  \draw (1.25,2) node [below] {$e'$};
  
  \draw (1.25,4.5) node [below] {$e''$};
  \draw[red] (2,3.9) node [below] {$p$};
  
  \draw[red] (2.7,1.4) node [below] {$q$};

  \end{tikzpicture}
 \caption{}
\end{subfigure}
\hfill
\begin{subfigure}[t]{0.4\linewidth}
\centering
\begin{tikzpicture}
\coordinate (2) at (1.5,0);
\coordinate (1) at (1.85-0.5,0.35);
\coordinate (3) at (1,3.45);
\coordinate (3i) at (1,2.55);

\coordinate (4) at (0.5,3.6);
\coordinate (4i) at (0.5,2.4);

\coordinate (p) at (0.2,4.1);
  \coordinate (q) at (0.9,2);
\coordinate (ip) at (0.2,2);
  \coordinate (iq) at (0.9,4.1);

 \foreach \i in {2,1,3,3i,4,4i}
 \draw[fill=black](\i) circle (0.15em);
 
\draw[red,fill=red](p) circle (0.15em);
  \draw[red,fill=red](q) circle (0.15em);
   
\draw[red,fill=red](ip) circle (0.15em);
  \draw[red,fill=red](iq) circle (0.15em);

\draw[thin, black] (2) to [out=90, in=40] (1.2-0.5,0.5) to [out=40+180, in=70] (1-0.5-0.5,0.2) to [out=70+180, in=65] (-0.5,0) ;
\draw[thin, black]  (-0.5,0)  to [out=65+180, in=140]  (-0.5,-0.5)  to [out=140+180, in=140] (1.3-0.5,-0.5) to [out=140+180, in=-90](2) ;

\draw[thin, black] (3i)--(3);
\draw[thin, black] (4i)--(4);

\draw[thin, black] (1.5,2) to [out=90, in=40] (0.7,2+0.5) to [out=40+180, in=70] (0,2+0.2) to [out=70+180, in=65] (-1+0.5,2) ;
\draw[thin, black]  (-1+0.5,2)  to [out=65+180, in=140]  (-1+0.5,2-0.5)  to [out=140+180, in=140] (0.8,2-0.5) to [out=140+180, in=-90](1.5,2) ;

\draw[thin, black] (1+0.5,4) to [out=90, in=40] (0.7,4+0.5) to [out=40+180, in=70] (0,4+0.2) to [out=70+180, in=65] (-1+0.5,4) ;
\draw[thin, black]  (-1+0.5,4)  to [out=65+180, in=140]  (-1+0.5,4-0.5)  to [out=140+180, in=140] (0.8,4-0.5) to [out=140+180, in=-90](1.5,4) ;

\draw[thin, black] (2) to [out=20, in=-60] (1.4+0.5,0.5) to [out=-60+180, in=50] (0.8+0.5,0.3) ;

  \draw (0.5,1.3) node [below] {$\downarrow$};
  
  \draw (1.8,0.1) node [below] {$e$};
  
  \draw (0.25,3.3) node [below] {$e'$};
  
  \draw (1.3,3.4) node [below] {$e''$};
\draw[red] (p) node [below] {$p$};
  
  \draw[red] (q) node [below] {$q$};
  \draw[red] (ip) node [below] {$ip$};
  
  \draw[red] (iq) node [below] {$iq$};
  \end{tikzpicture}
 \caption{}
\end{subfigure}
\hfill
 \caption[below]{ }
  \label{figpic9}
 \end{figure}
\end{example}

We now deal with the case where a non-bridge edge lifts to a disconnecting pair.  

\begin{lemma}\label{lem:AbelPrymSeparatesTwoConnectedComponents2}
    Suppose that an edge $e$ of $\Gamma$ is not a bridge but its preimages are a disconnecting pair. Let $p$ and $q$ be distinct points of $\wt\Gamma$ where  $\Psi$ is finite and  $\Psi(p) = \Psi(q)$. Then $p$ and $q$ belong to distinct components of the graph obtained by removing the preimages of $e$. 
    \end{lemma}

\begin{proof}
Since $\Psi$ is finite at $p$ and $q$ and contracts the preimages $e',e''$ of $e$, it follows that $p$ and $q$ are in their complement. Let $\phi$ be the function such that $p + \iota q + \ddiv \phi = q + \iota p$ and $A$ the set where the minimum of $\phi$ is obtained.  If we assume for the sake of contradiction that $p$ and $q$ belong to the same  component then, as in the proof of \cref{lem:AbelPrymSeparatesTwoConnectedComponents}, we see that $\phi$ is constant on both $e'$ and $e''$. It then follows that $p\simeq \iota q$, which can only happen if they both belong to a contractible component of $\wt\Gamma$, which is a contradiction.

\end{proof}

The following lemma is a step towards showing that the fibre of the Abel--Prym map, when finite, consists of at most $2$ points. 

\begin{lemma}\label{lem:disconnectingPairs}
    Suppose that $e_1,e_2,e_3,e_4$ are distinct non-bridge edges of a graph $\Gamma$ such that
    $\{e_1,e_2\}, \{e_2,e_3\},$ and $\{e_3,e_4\}$ are disconnecting pairs. Then $\{e_4,e_1\}$ is a disconnecting pair as well. 
\end{lemma}

\begin{proof}
By assumption, when removing $e_1,e_2,$ and $e_3$ we are left with three connected components. Denote $A$ the one adjacent to the edges $e_1,e_3$, denote $B$ the one adjacent to the edges $e_1,e_2$, and denote     $C$ the one adjacent to the edges $e_2,e_3$. 
Since $\{e_3,e_4\}$ is a disconnecting pair as well, it follows that $e_4$ is a bridge of the graph obtained by removing $e_3$. In particular, it is a bridge of either $A,B$, or $C$. 

Now consider the graph $\Gamma'$ obtained by removing $e_1$ from $\Gamma$. Since we already know that $e_4$ is a bridge in one of $A,B,$ or $C$, it is a bridge of $\Gamma'$. In particular, $\{e_1,e_4\}$ is a disconnecting pair. 
\end{proof}

Alternatively, the lemma can be proven from the matroidal perspective as follows: any disconnecting pair is a 2-circuit of the co-graphic matroid. The result now follows since the relation of being part of a 2-circuit is transitive.

\begin{proposition}
    \label{prop:fibreSize}
    Let $\pi:\wt\Gamma\to\Gamma$ be a free double cover and suppose that $\Psi(p) = \Psi(q) = \Psi(r)$, where $p,q,r\in\wt\Gamma$ are points that are not adjacent to bridges and are not in the interior of the $\Psi$-contractible locus. 
     Then at least two of the points $p,q,r$ coincide. In particular, every fibre of the Abel--Prym map is either infinite or consists of at most two points.

\end{proposition}

\begin{proof}
For the sake of contradiction, suppose that $p,q,$ and $r$ are distinct. In particular, we can choose a model for the graph where they belong to distinct edges. Assume first that they belong to the interior of those edges. Denote $e_1,e_2,e_3,e_4$  the edges containing $p,\iota q,r,\iota p$ respectively. 
    Then, since $p+\iota q\simeq q+\iota p$, it follows that $\{e_1,e_2\}$ is a disconnecting pair. Similarly, $\{e_2, e_3\}$ and $\{e_3, e_4\}$ are disconnecting pair. It now follows from \cref{lem:disconnectingPairs} that $\{e_4,e_1\}$ is a disconnecting pair. 
    
    Note that, since $\iota p\in e_4$, we have $e_4 = \iota e_1$. Let $e = \pi(e_1)$. If $e$ is not a bridge then, as the number of connected components of $\wt\Gamma\setminus\pi^{-1}(e)$ is larger than that of $\Gamma\setminus\ e$, it follows from \cref{lem:contractedEdges} that $\Psi$ contracts the edges $e_4$ and $e_1$, which is a contradiction. If $e$ is a bridge then, by the finiteness of $\Psi$,  it follows that $e_1$ and $e_4$ are non-bridges,  so by \cref{lem:edgeMappedWithMultiplicity2}, $\Psi$ has multiplicity $2$ at $p$ and is the unique point in the fibre of the Abel--Prym map over $\Psi(p)$, which, again, is a contradiction.

    The situation is not so different when we don't assume that the points necessarily belong to the interior of edges. Assume first that no edge adjacent to $p$ gets contracted by the Abel--Prym map. 
    In this case,  the relation $p+\iota q\simeq q+\iota p$ still implies that there are edges $e_1,e_2$ adjacent to $p,\iota q$ that form a disconnecting pair. Similarly, there are edges $e_3,e_4$ adjacent to $r$ and $\iota p$ so that $\{e_2,e_3\}$ and $\{e_3,e_4\}$ are disconnecting pairs, so $\{e_4,e_1\}$ is a disconnecting pair as well. Now a similar argument to the one above shows that $\Psi$ is either infinite or has multiplicity $2$ at $p$, which is a contradiction.

Finally, suppose that $p$ is adjacent to a non-bridge edge $e'$ that is contracted by $\Psi$. Then $\pi(e')$ is a bridge $e$ of $\Gamma$ and $\pi^{-1}(e)$ is an edge $e''$ that forms a properly disconnecting pair with $e'$. Denote $\wt\Gamma_1$ the connected component of $\wt\Gamma\setminus \{e',e''\}$ containing $p$ and $\wt\Gamma_2$ the other connected component. 
From \cref{lem:AbelPrymSeparatesTwoConnectedComponents2}, the points $q$ and $r$ belong to $\wt\Gamma_2$. But then they belong to the same connected component, contradicting \cref{lem:AbelPrymSeparatesTwoConnectedComponents2}. 
\end{proof}

\begin{remark}
 The statement of
 \cref{prop:fibreSize} may fail to hold for points that are adjacent to bridges, since all the boundary points of the bridge locus are mapped via Abel--Prym to the same point. 
As one can see in \cref{fig:bridges},  if  $p,q,$ and $r$  distinct points in the source graph adjacent to the three  bridges and to the loops, then they are all boundary points of the contracted locus that are identified by the Abel--Prym map.

     \begin{figure}[h]
                \centering\begin{tikzpicture}

\coordinate (c) at (0,0) {};

\coordinate (n1) at (1,0.7) {};
\coordinate (n2) at (1,0) {};
\coordinate (n3) at (1,-0.7) {};

\draw (c) -- (n1);
\draw (c) -- (n2);
\draw (c) -- (n3);

\draw (1.3,0.7) circle (0.3);
\draw (1.3,0) circle (0.3);
\draw (1.3,-0.7) circle (0.3);

\draw (-0.3,0) circle (0.3);

\coordinate (c1) at (0,0+5) {};

\coordinate (n11) at (1,5.7) {};
\coordinate (n21) at (1,5) {};
\coordinate (n31) at (1,-0.7+5) {};

\coordinate (c2) at (0,0+2.5) {};

\coordinate (n12) at (1,0.7+2.5) {};
\coordinate (n22) at (1,2.5) {};
\coordinate (n32) at (1,-0.7+2.5) {};
\coordinate (x) at (4,3.75) {};

\draw (c1) -- (n11);
\draw (c1) -- (n21);
\draw (c1) -- (n31);
\draw (c2) -- (n12);
\draw (c2) -- (n22);
\draw (c2) -- (n32);

\foreach \i in {c,x,n1,n2,n3,c1,n11,n21,n31,c2,n12,n22,n32}
\draw[fill=black](\i) circle (0.15em);

\draw (1.3,5.7) circle (0.3);
\draw (1.3,5) circle (0.3);
\draw (1.3,-0.7+5) circle (0.3);

\draw (1.3,0.7+2.5) circle (0.3);
\draw (1.3,2.5) circle (0.3);
\draw (1.3,-0.7+2.5) circle (0.3);

\

\node at (0.5,1.3) {$\downarrow$};
\node at (2.4,3.75) {$\rightarrow$};
\node at (2.4,4) {$\Psi$};

 \draw[thin]   (c1) to [out=-60, in=60] (c2);
				\draw[thin]   (c1) to [out=180+60, in=90+30] (c2);

\draw[thin]   (x) to [out=90+60, in=180] (4,4.2);
\draw[thin]   (4,4.2) to [out=0, in=30] (x);
\draw[thin]   (x) to [out=180+90, in=-90] (3.57,3.75);
\draw[thin]   (x) to [out=120, in=90] (3.57,3.75);
\draw[thin]   (x) to [out=120+180, in=-90] (4.45,3.75);
\draw[thin]   (x) to [out=60, in=90] (4.45,3.75);
\end{tikzpicture}
                \caption{A double cover whose $\Psi$-collapsible locus consists of the six bridges and a disconnecting pair of edges.}
                \label{fig:bridges}
            \end{figure}
\end{remark}

We now specialise to the case where the source graph of the double cover is hyperelliptic.

\begin{lemma}\label{shapeOfFibres}
   Suppose that $\wt\Gamma$ is hyperelliptic. If $p\in\wt\Gamma$ and $q = \iota\tj p$ then $\Psi(p) = \Psi(q)$.   Conversely, if  $\Psi(p) = \Psi(q)$ for  $p,q\in\wt\Gamma$ 
   that are not adjacent to bridges and are not in the interior of the $\Psi$-contractible locus,
then either $q = p$ or $q = \iota\tj p$.
\end{lemma}
\begin{proof}
   
    Suppose first that  $q= \iota\tj p$. Then, since $q + \tj q\simeq p + \tj p$, it follows that $p - \tj q \simeq q - \tj p$.  But, since $\tj$ commutes with $\iota$, we have $\tj q = \tj\iota\tj p = \iota p$ and $\tj p = \iota q$, so $\Psi(p) = p-\iota p\simeq q-\iota q = \Psi(q)$.

 Conversely, suppose that $\Psi(q) = \Psi(p)$ and $p$ and $q$ are as in the statement. Denote $p' =\iota\tj p$. From the first part, $\Psi(p') = \Psi(p)$, so  \cref{prop:fibreSize} implies that either $p' = q$ or $p' = p$.  In the former case, 
  we have $q = p' = \iota\tj p$, so 
 we are done.  For the latter, by properties of the hyperelliptic involution, we have $p + \tj p\simeq \iota p + \tj\iota p$. But since $p'=p$, we have $p + \tj p\simeq \iota p + p$, so we conclude $\tj p\simeq \iota p$. Applying $\iota$ to both sides, we conclude $p\simeq q$. But we assumed that neither  $p$ nor $q$  is adjacent to a bridge, so it must be that $p=q$. 
\end{proof}

We are ready to prove one direction of \cref{maintheorem:degreeOfAbelPrym}.

\begin{theorem}\label{thm:ifHyperelliptic}
    Let $\pi\colon \tGa\to \Gamma$ be a  free double cover of metric graphs of genus $g_{\Gamma}\geq 2$, with $\wt\Gamma$ hyperelliptic and 2-edge-connected. Then the Abel--Prym map  $\Psi\colon \tGa\to \,\Prym(\tGa/\Gamma)$ is a harmonic morphism of degree 2.

    \begin{proof}

Let $p\in\wt\Gamma$.
     As the multiplicity of the Abel--Prym map is constant along edges, we may assume that $\val(p)\geq 3$.
If $\Psi$ contracts a neighbourhood of $p$, then for every edge emanating from $\Psi(p)$, there is no edge adjacent to $p$ mapping to it, so harmonicity holds vacuously. We may therefore assume that at least one edge emanating from $p$ is not contracted. From \cref{prop:fibreSize}, there is at most one other point $q$ such that $\Psi(q) = \Psi(p)$ and $q$ is not in the interior of the $\Psi$-contracted locus.

Suppose first that there is exactly one other such point $q\neq p$. From \cref{shapeOfFibres}, we have $q = \iota\tj p$. 
Let $e$ be an edge adjacent to $\Psi(p)$. Then there is an edge $f'$ adjacent to either $p$ or $q$ mapping to it. Without loss of generality, assume that $f'$ is adjacent to $p$. As $\iota\circ\tj$ is an automorphism, there is an edge $f''$ emanating from $q$  that maps to $e$ as well.
Since  the preimage of every edge consists of at most 2 edges and $\iota$ does not flip edges, 
it follows that no other edge adjacent to $p$ maps to $e$.
We conclude that the preimage of every edge adjacent to $\Psi(p)$ at a neighbourhood of $p$ is a single edge mapping with no dilation, so $\Psi$ is harmonic at $p$.

Now suppose  that $p$ is the unique point  mapping to $\Psi(p)$ which is not in the interior of the $\Psi$-collapsible locus.
Let $e$ be an edge adjacent to $\Psi(p)$. Then there is some edge $f$ of $\wt\Gamma$ mapping to it and, by continuity and uniqueness, that edge must terminate at $p$.   If $f$ is being dilated then it is the only edge mapping to $e$. If, on the other hand, 
$\Psi$ does not dilate $f$, let $f' =\iota(\tj f)$. We will show that $f'$ is distinct from $f$. Suppose for the sake of contradiction that $f'=f$ and let $q$ be a point on $f$.   Then $\pi(q)$ is a fixed point of $j$ since 
       \[
       j(\pi(q)) = \pi(\tj (q))  = \pi(\iota q) = \pi(q),
       \]
       where we use \cref{commutative} to show that $j\pi = \pi\tj$.
       In particular, $\pi(q)$ is part of a bridge. But then $\Psi$ either contracts $f$ or dilates it, which is a contradiction.  So $f'$ is distinct from $f$ and, since $p$ is the only point of $\wt\Gamma$ mapping to $\Psi(p)$, the edge $f'$ must be adjacent to $p$ as well.
We conclude that for every edge $f$ adjacent to $p$, its preimage consists of either a single dilated edge or two non-dilated edges, so $\Psi$ is harmonic at $p$.

\end{proof}
\end{theorem}

Note that, when removing the assumption that $\wt\Gamma$ is 2-connected, the statement of the theorem no longer holds. 

\begin{example}\label{ex:bridgesRuinEverything}
Consider the graph seen in \cref{fig:bridges}. Let $v$ be a vertex of $\wt\Gamma$. The loop adjacent to $v$ maps with multiplicity 1 onto one of the loops adjacent to $\Psi(v)$. However, the preimage of each of the other loops adjacent to $\Psi(v)$ is empty and, in particular, $\Psi$ is not harmonic. 

 It's insightful to note that, after contracting the bridges of $\wt\Gamma$, we get a double cover of hyperelliptic graphs and the Abel--Prym map  becomes harmonic.

\end{example}

\begin{example}
 If contracting the cyclic $\Psi$-collapsible locus results in a hyperelliptic graph, it doesn't mean that the original graph was hyperelliptic. 
  \cref{fig:counterexampleofdeg2} depicts a non-hyperelliptic double cover of the chain of three loops. The Abel--Prym map contracts the bridges and it does not satisfy the definition of harmonicity on the points in the bridges. 

     \begin{figure}[H]
    \centering
 \begin{tikzpicture}
\coordinate (1) at (6,3.5);          
\coordinate (1x) at (1.5-2+0.5,0);
\coordinate (3x) at (-4,0);
\coordinate (a1x) at (1.5-4+2.5,2);
\coordinate (a3x) at (-4,2);
\coordinate (b1x) at (1.5-4+2.5,5);
\coordinate (b3x) at (-4,5);
\coordinate (2x) at (-2.5,0);
\coordinate (4x) at (-1.5,0);
						\coordinate (a2x) at (-2.5,2);
					\coordinate (a4x) at (-1.5,2);
						\coordinate (b2x) at (-2.5,5);
					\coordinate (b4x) at (-1.5,5);
				
				\foreach \i in {1x,2x,4x,3x,a1x,1,a2x,a4x,b2x,b4x,a3x,b1x,b3x}
				\draw[fill=black](\i) circle (0.15em);
				\draw[thin, black] (3x)--(2x);
					\draw[thin, black] (4x)--(1x);
				\draw[thin, black] (a4x)--(-1.95,3.3);
				\draw[thin, black] (-2.09,3.65)--(b2x);
				\draw[thin, black] (b4x)--(a2x);

\node[circle, draw, minimum size=1cm] (c) at (5.5,3.5) {};
\node[circle, draw, minimum size=1cm] (c) at (6.5,3.5) {};

				\node[circle, draw, minimum size=1cm] (c) at (-0.5-4,0) {};
                \node[circle, draw, minimum size=1cm] (c) at (-0.5-4,2) {};
                \node[circle, draw, minimum size=1cm] (c) at (-0.5-4,5) {};
				\node[circle, draw, minimum size=1cm] (c1) at (2-4+2.5,0) {};
					\node[circle, draw, minimum size=1cm] (c3) at (-2,0) {};
				\node[circle, draw, minimum size=1cm] (c3) at (0.5,2) {};
                \node[circle, draw, minimum size=1cm] (c3) at (0.5,5) {};
				\draw[thin, black] (a3x)--(a1x);
				\draw[thin, black] (b3x)--(b1x);

				\draw (-2,1.5) node [below] {$\downarrow$};
                \draw (-2,1.2) node [right] {$\pi$};
				\draw (3,3.5) node [below] {$\rightarrow$};
				\draw (3,3.5) node [above] {$\Psi$};
			
					\end{tikzpicture}
  \caption{A non-harmonic Abel--Prym map $\Psi$. }
    \label{fig:counterexampleofdeg2}
 \end{figure}

\end{example}

Next, we prove the converse to \cref{thm:ifHyperelliptic}.
\begin{theorem}\label{thm:harmonicImpliesHyperelliptic}
    Suppose that $\wt\Gamma$ is 2-edge-connected and that
    $\Psi$ is a  harmonic morphism of degree 2. Then $\wt\Gamma$ is hyperelliptic.

\end{theorem}

\begin{proof}
    We need to define a map $\tj$ on $\wt\Gamma$ and prove that it is a hyperelliptic involution.   Let $p\in\wt\Gamma$ that is not in the $\Psi$-collapsible locus.  If there is $q\neq p$ such that $\Psi(q)=\Psi(p)$, define  $\tj(p) = \iota q$. Otherwise, define $\tj(p) = \iota p$.  Now, suppose that $p$ belongs to the cyclic $\Psi$-collapsible locus.  Then $p$ belongs to an edge $e'$ that forms, together with its involution $e''$, a properly disconnecting pair, and their image in $\Gamma$ is not a bridge. 
 Identify $e'$ and $e''$ with the intervals $I'=[0,\ell]$ and $I''=[0,\ell]$ (where $\ell$ is the length of $e$), so that $\iota$ corresponds to the identity morphism between the intervals, and let $x\in I'$ be the point identified with  $p$. Now define   $\tj(p)$ to be the point on the edge $e''$ corresponding to $\ell-x\in I''$.   For future reference, denote $q$ the point corresponding to $x$ on $I''$.

    We first argue  that $\tj$ is  continuous. 
    
    Clearly, the map is continuous along edges. Now let $v'$ be a vertex of $\wt\Gamma$ that is not in the interior of the $\Psi$-collapsible locus. Let $e'$ and $f'$ be edges emanating from $v'$ that are not being contracted by $\Psi$. If $\Psi$ is globally injective on both, then, by definition, $\tj$ coincides with $\iota$ and is clearly continuous.  If $\Psi$ is globally injective on neither, then there is another edge $e''$ mapping to $\Psi(e')$ with a vertex $v''$ mapping to $\Psi(v')$. By harmonicity of $\Psi$ at $v'$, there is an edge adjacent to $e''$ mapping to $\Psi(f')$. It follows that that the limits of $\tj$ on both tangent directions from $v'$ coincide. 
    Finally, suppose that $\Psi$ is globally injective on $e'$ but not on $f'$. Then, since the degree is $2$, it follows that $\Psi$ dilates $e'$ by a factor of $2$. By harmonicity at $v'$,  there is an edge $f''$ emanating from $v$ such that $\Psi(f'') = \Psi(f')$. It follows that  the limits of $\tj$ from the various tangent directions $v'$ coincide with $\tj(v')$.

   In order to prove that $\tj$ is a hyperelliptic involution,  it suffices to show that, whenever $p$ is in the interior of an edge, the divisor $p+\tj p$ moves. 
   By definition, for every $p$, we have $p + \tj p\simeq q + \tj q$, where $q$ is as above. If $p + \tj p \neq q + \tj q$ then we are done. Otherwise, if $q\neq p$, then $p = \iota p$, a contradiction. If $q = p$  
    then $p$ and $\iota p$  form a disconnecting pair so $p+\iota p$ moves along the entire edge. But $\iota p = \tj p$ in this case, so we are done.

   \end{proof}

\subsection{The non-hyperelliptic case}
We finish the section with   examples  illustrating the behaviour of the Abel--Prym map when the source graph is not hyperelliptic. 
In the algebraic case, if the source curve is not hyperelliptic then the Abel--Prym map is an embedding (\cite{BL04},\cite{LZ22}).
However, in the tropical case, when the tropical Abel--Prym map is not harmonic of degree 2, it does not necessarily follow that it has degree 1 everywhere injective. It could have degree 1 on parts of the graph and degree 2 elsewhere.

\begin{example}\label{ex:hyperellipticType}
Consider the double cover of metric graphs depicted in \cref{hyptype}. The source graph is not hyperelliptic and the fibres of the Abel--Prym map consist of either  $1,2,$ or infinitely many points. Indeed, for any point $p$ as in the figure, there is a unique corresponding point $q\neq p$ such that $[p-\iota(p)] = [q-\iota(q)]$, so $\Psi$ coincides on the edges containing $p$ and $q$. The dashed edges get contracted, and $\Psi$ is injective with multiplicity 1 everywhere else. It follows that $\Psi$ doesn't satisfy the definition of harmonicity at the vertices $v_1$ and $v_2$.

    \begin{figure}[ht]
    \centering
\begin{tikzpicture}
\coordinate (2) at (2,1.5);
\coordinate (3) at (1,1.5);
\coordinate (12) at (-1,1.5);
\coordinate (a2) at (2,5);
\coordinate (a3) at (1,5);
\coordinate (a12) at (-1,5);
\coordinate (b2) at (2,8);
\coordinate (b3) at (1,8);
\coordinate (b12) at (-1,8);
\coordinate (b4) at (1.4,8.7); 
\coordinate (a4) at (1.4,4.3);
\coordinate (b5) at (-0.4,8.7);
\coordinate (a5) at (-0.4,4.3);
\coordinate (c2) at (8,6.5);
\coordinate (c3) at (7,6.8);
\coordinate (c3') at (6,6.2);
\coordinate (c12) at (5,6.5);
 
 \foreach \i in {2,3,12,a2,a3,a12,b2,b3,b12,c2,c3,c3',c12}
 \draw[fill=black](\i) circle (0.15em);

\draw[red,fill=red](b4) circle (0.20em);
\draw[red,fill=red](b5) circle (0.20em);
\draw[red](a4) circle (0.20em);
\draw[red](a5) circle (0.20em);
 \draw[thin, black] (12)--(3);
 \draw[thin, black] (a12)--(a3);
 \draw[thin, black] (b12)--(b3);
 \draw[thin, black] (12) to [out=90, in=90] (2);
\draw[thin, black] (12) to [out=90+180, in=-90] (2);
\draw[thin, black] (3) to [out=90, in=90] (2);
\draw[thin, black] (3) to [out=270, in=-90] (2);
\draw[thin, black] (b12) to [out=90, in=90] (b2);
\draw[thin, black] (a12) to [out=270, in=-90] (a2);
\draw[thin, black] (b3) to [out=90, in=90] (b2);
\draw[thin, black] (b3) to [out=270, in=-90] (b2);
\draw[thin, black] (a3) to [out=90, in=90] (a2);
\draw[thin, black] (a3) to [out=270, in=-90] (a2);

\draw[thin, black] (c12) to [out=90, in=90] (c2);
\draw[thin, black] (c12) to [out=45, in=170] (c3);
\draw[thin, black] (c3') to [out=-20, in=180+45] (c2);
\draw[thin, black] (c3) to [out=50, in=90] (c2);
\draw[thin, black] (c3) to [out=300, in=180] (c2);
\draw[thin, black] (c12) to [out=40, in=90+45] (c3');
\draw[thin, black] (c12) to [out=300, in=180+45] (c3');

  \draw[thin, dashed, black] (b12) to [out=-40, in=-260] (a2);
  
\draw[thin, dashed, black] (a12) to [out=40, in=180+45] (1,6.4);
    \draw[thin, dashed, black] (1.2,6.6) to [out=45, in=260] (b2);
  \draw (0.5,3.5) node [below] {$\downarrow$};
   \draw (0.8,3.4) node [below] {$\pi$};
   \draw (b4) node [above] {$p$};
   \draw (b5) node [above] {$\iota(q)$}; 
   \draw (a4) node [below] {$\iota(p)$};
   \draw (a5) node [below] {$q$};
\draw (3.5,6.4) node [above] {$\Psi$};
   \draw (3.5,6.5) node [below] {$\rightarrow$};
   \draw (b2) node [right] {$v_2$};
   
    \draw (b12) node [below] {$v_1$};
   \end{tikzpicture}
 \caption{A free double cover of a non-hyperelliptic metric graph with points $p$ and $q$ such that $p-\iota(p)\simeq q-\iota(q)$.}
    \label{hyptype}
\end{figure}
\end{example}

\begin{remark}\label{rem:hyperellipticType}
    The source graph in \cref{ex:hyperellipticType}
    is of hyperelliptic type but not hyperelliptic, meaning that its Jacobian is isomorphic to the Jacobian of a hyperelliptic graph (see \cite{C} for a detailed account of such graphs). 
  It is natural to ask whether the non-injectivity of the Abel--Prym map only happens for graphs of hyperelliptic type. However, this is not the case as exemplified in the double cover appearing in \cref{fig:nonhyptype}.
\end{remark}

\begin{remark} 
From \cref{ex:hyperellipticType}, we see that the Abel--Prym map is not always  harmonic  onto its image. Indeed, although the map identifies pairs of edges, there are also edges where the map is injective and has multiplicity $1$, contradicting the definition of harmonicity.
However, it is a parameterized tropical curve in the sense of \cite[Definition 3.1]{Blomme_CurvesInAbelian1}.
\end{remark}

            \begin{figure}[h]
                \centering\begin{tikzpicture}
				
				\coordinate (1x) at (1.5-2+0.5,0);
				\coordinate (3x) at (-4,0);
				\coordinate (a1x) at (1.5-4+2.5,2);
				\coordinate (a3x) at (-4,2);
				\coordinate (b1x) at (1.5-4+2.5,5);
				\coordinate (b3x) at (-4,5);
					\coordinate (2x) at (-2.5,0);
					\coordinate (4x) at (-1.5,0);
						\coordinate (a2x) at (-2.5,2);
					\coordinate (a4x) at (-1.5,2);
						\coordinate (b2x) at (-2.5,5);
					\coordinate (b4x) at (-1.5,5);
				
				\foreach \i in {1x,2x,4x,3x,a1x,a2x,a4x,b2x,b4x,a3x,b1x,b3x}
				\draw[fill=black](\i) circle (0.15em);
				\draw[thin, black] (3x)--(2x);
					\draw[thin, black] (4x)--(1x);
				\draw[thin, black] (a4x)--(-1.95,3.3);
				\draw[thin, black] (-2.09,3.65)--(b2x);
				\draw[thin, black] (b4x)--(a2x);

				\node[circle, draw, minimum size=1cm] (c) at (-0.5-4,0) {};
				\node[circle, draw, minimum size=1cm] (c1) at (2-4+2.5,0) {};
					\node[circle, draw, minimum size=1cm] (c3) at (-2,0) {};
				
				\draw[thin, black] (a3x)--(a1x);
				\draw[thin, black] (b3x)--(b1x);

				\draw (-2,1.5) node [below] {$\downarrow$};
				
				\draw[thin, black] (b3x) to [out=-45, in=45] (a3x);
				\draw[thin, black] (b3x) to [out=180+45, in=90+45] (a3x);
				\draw[thin, black] (b1x) to [out=-45, in=45] (a1x);
				\draw[thin, black] (b1x) to [out=180+45, in=90+45] (a1x);
			
					\end{tikzpicture}
                \caption{A graph of non-hyperelliptic type covering a hyperelliptic graph  where the Abel--Prym map is not harmonic.}
                \label{fig:nonhyptype}
            \end{figure}

Next, we show that, while injectivity of the Abel--Prym map is not guaranteed in the non-hyperelliptic case, it can still happen.

\begin{example}
   \cref{notinnerelliptic} shows a   double cover of graphs, obtained by gluing together two copies $\wt\Gamma_1$ and $\wt\Gamma_2$  of the graph from  \cref{ex:hyperellipticType}. The vertical edges of $\wt\Gamma$ are no longer properly disconnecting pairs, so they don't get contracted by the Abel--Prym map. Furthermore, from \cref{thm:behaviousOnEdges}, the multiplicity of $\Psi$ is always 1. From \cref{lem:AbelPrymSeparatesTwoConnectedComponents}, if $\Psi(p) = \Psi(q)$ for some points $p,q\in\wt\Gamma$, then $p$ and $q$ both belong to either $\wt\Gamma_1$ or $\wt\Gamma_2$. By exhausting all the possibilities, we see that $p$ and $q$ must coincide, so $\Psi$ is injective of degree 1, namely a graph isomorphism onto its image.

    \begin{figure}[h]
    \centering
\begin{tikzpicture}
	
	\coordinate (2) at (2,1.5);
	\coordinate (3) at (1,1.5);
       \coordinate (12) at (-1,1.5);
	\coordinate (a2) at (2,5);
	\coordinate (a3) at (1,5);
       \coordinate (a12) at (-1,5);
	\coordinate (b2) at (2,8);
	\coordinate (b3) at (1,8);
       \coordinate (b12) at (-1,8);
 \coordinate (4) at (2,1.5);
 \coordinate (a4) at (2,5);
 \coordinate (b4) at (2,8);

 \coordinate (2i) at (5,1.5);
	\coordinate (3i) at (4,1.5);
       \coordinate (12i) at (2,1.5);
	\coordinate (a2i) at (5,5);
	\coordinate (a3i) at (4,5);
       \coordinate (a12i) at (2,5);
	\coordinate (b2i) at (5,8);
	\coordinate (b3i) at (4,8);
       \coordinate (b12i) at (2,8);
       
 \foreach \i in {2,3,12,a2,4,a3,a4,a12,b2,b3,b4,b12,2i,3i,12i,a2i,a3i,a12i,b2i,b3i,b12i}
 \draw[fill=black](\i) circle (0.15em);
\draw[thin, black] (2)--(4);
\draw[thin, black] (a2)--(a4);
\draw[thin, black] (b2)--(b4);
 \draw[thin, black] (12)--(3);
 \draw[thin, black] (a12)--(a3);
 \draw[thin, black] (b12)--(b3);
 \draw[thin, black] (12) to [out=90, in=90] (2);
\draw[thin, black] (12) to [out=90+180, in=-90] (2);
\draw[thin, black] (3) to [out=90, in=90] (2);
\draw[thin, black] (3) to [out=90+180, in=-90] (2);
\draw[thin, black] (b12) to [out=90, in=90] (b2);
\draw[thin, black] (a12) to [out=90+180, in=-90] (a2);
\draw[thin, black] (b3) to [out=90, in=90] (b2);
\draw[thin, black] (b3) to [out=90+180, in=-90] (b2);
\draw[thin, black] (a3) to [out=90, in=90] (a2);
\draw[thin, black] (a3) to [out=90+180, in=-90] (a2);

  \draw[thin, black] (b12) to [out=-40, in=-260] (a2);
  
\draw[thin, black] (a12) to [out=40, in=180+45] (1,6.4);
    \draw[thin, black] (1.2,6.6) to [out=45, in=260] (b2);

 \draw[thin, black] (12i)--(3i);
 \draw[thin, black] (a12i)--(a3i);
 \draw[thin, black] (b12i)--(b3i);
 \draw[thin, black] (12i) to [out=90, in=90] (2i);
\draw[thin, black] (12i) to [out=90+180, in=-90] (2i);
\draw[thin, black] (3i) to [out=90, in=90] (2i);
\draw[thin, black] (3i) to [out=90+180, in=-90] (2i);
\draw[thin, black] (b12i) to [out=90, in=90] (b2i);
\draw[thin, black] (a12i) to [out=90+180, in=-90] (a2i);
\draw[thin, black] (b3i) to [out=90, in=90] (b2i);
\draw[thin, black] (b3i) to [out=90+180, in=-90] (b2i);
\draw[thin, black] (a3i) to [out=90, in=90] (a2i);
\draw[thin, black] (a3i) to [out=90+180, in=-90] (a2i);
  \draw[thin, black] (b12i) to [out=-40, in=-260] (a2i);
\draw[thin, black] (a12i) to [out=40, in=180+45] (4,6.4);
    \draw[thin, black] (4.2,6.6) to [out=45, in=260] (b2i);

  \draw (2,3.5) node [below] {$\downarrow$};
  
  \end{tikzpicture}
  
 \caption{Example where the Abel--Prym is injective everywhere.}
    \label{notinnerelliptic}
\end{figure}
\end{example}

\section{Free double covers of hyperelliptic graphs}\label{sec:hyperellipticDoubleCovers}

{  
In this section we exhibit an explicit construction of the free double covers of hyperelliptic graphs, leading to a strong characterisation for when such covers are themselves hyperelliptic, as well as an enumeration of all such covers. 
A key role in the construction will be played by the fixed points of the hyperelliptic involution. 
Throughout the section  $\pi\colon \tGa\to \Gamma$ will be a free double cover of metric graphs with involution $\iota$ on $\tGa$ exchanging the fibres and hyperelliptic involution $j$ on $\Ga$. Note that $\tGa$ is not assumed to be hyperelliptic.

\subsection{The fixed points of the hyperelliptic involution}\label{sec:fixedpoints}
    The hyperelliptic involution  $j$ of a 2-edge-connected metric graph $\Gamma$ of genus $g$ has $g+1$ fixed points. Indeed, since $j$ is a non-degenerate harmonic morphism of degree $2$ from $\Gamma$ to a tree $T$, the Riemann-Hurwitz formula \cite[Section 2.10]{ABBR} implies that
 the degree of the ramification divisor is given by
\begin{equation}\label{numberOfFixedPoints}
        \deg(R_j)=\deg(K_{\Gamma})-\deg(j^*(K_T)).
    \end{equation}
    Since $\deg j^*(K_T)=2\deg(K_T)=-4$, the ramification divisor has degree $2g+2$. But each ramification point has multiplicity 2, so there are $g+1$ of them, see \cite[Proposition 9]{Panizzut}.

\begin{remark}
 Under the assumption that hyperelliptic graphs don't have any leaf-edges, every cut point of $\Gamma$ is fixed by the hyperelliptic involution. 

\end{remark}

\begin{remark}\label{rem:onlytwochoicesfortheinvolution}
By \cref{commutative}, the preimages via $\pi$ of a fixed point $x\in \Ga$ with respect to $j$, are either fixed  or exchanged by the involution $\tj$. Indeed, if $p$ and $\iota(p)$ are in the fibre $\pi^{-1}(x)$, then 
 $\pi(\tj p) = j \pi (p) = j x = x$, so either $\tj p = p$  or $\tj p=\iota p$. 
\end{remark}

\begin{proposition}\label{prop:valencyoffixedpoints}
    Let $\Gamma$ be a $2$-edge connected hyperelliptic graph  with hyperelliptic involution $j$ and let $p$ be a fixed point of $j$. Then the valency of $p$ equals $2k$, where $k$ is the number of components of $\Gamma\setminus{p}$.
\end{proposition}
\begin{proof}
Let $T$ be the quotient tree $\Gamma/j$ and $q$ the image of $p$ in $T$. By the assumption that there are no bridges, every edge adjacent to $q$ lifts to two edges adjacent to $p$ and, therefore, the valency of $p$ is even. Furthermore, the valency of $q$ equals the number of components of $\Gamma\setminus{p}$, so the result follows.

\end{proof}

\begin{proposition}\label{fixedpoints}
Suppose that $\Gamma$ and $\wt\Gamma$ are hyperelliptic with involutions  $\tj$ and $j$ respectively.  Let $x\in \Ga$ be a fixed point of $j$, and let $\{p,\iota(p)\}=\pi^{-1}(x)$. Then, 
    \begin{itemize}
        \item[i)] If $x$ is a cut vertex, then $p$ is a fixed point of $\tj$ $\Longleftrightarrow $ $p$ and $\iota(p)$ are each cut vertices;
        \item[ii)]  If $x$ is not a cut vertex, then $p$ is fixed by $\tj$ $\Longleftrightarrow $ $p$ and $\iota(p)$ are not a disconnecting pair.
    \end{itemize}

    \noindent Furthermore, if $\Gamma$ is 2-edge-connected, then $\exists!$  point $x$ fixed by $j$  whose preimages in $\tGa$ are not fixed by $\tj$, in which case $\tj(p)=\iota(p).$ 
    \begin{proof}    Let us contract all the bridges of $\Gamma$ and their preimages in $\tGa$, which may or may not be bridges, so that they are 2-edge-connected metric graphs. This does not affect whether $x$ is a cut vertex or any of the other conditions appearing in the statement. 
    The image of each bridge under the contraction is  a cut vertex. 
        \begin{itemize}
            \item [i)] Suppose that $x$ is a cut vertex. Then
            \begin{itemize}
                \item [$(\Leftarrow)$] If $p$ is a cut vertex, then it follows from \cite[Lemma 3.10]{CH} that $p$ is fixed by $\tj$;
                
                \item [$(\Rightarrow)$] If $p$ is a fixed point of $\tj$ then, from \cref{prop:valencyoffixedpoints} applied to both $p$ and $x$, we conclude that the valency of $p$ is at least $4$. Therefore, by applying \cref{prop:valencyoffixedpoints} to $p$ and $\iota p$, they are both  cut vertices.  
            \end{itemize}
            \item[ii)] Suppose now that $x$ is not a cut vertex (but still a fixed point of $j$). By  \cref{prop:valencyoffixedpoints}, its valency is 2 and therefore $val(p)=val(\iota(p))=2$.
            \begin{itemize}
                \item [$(\Leftarrow)$] If $p$ is not fixed by $\tj$ then, by \cref{rem:onlytwochoicesfortheinvolution}, the divisor $p+\iota(p)$ has rank $1$. 
                Therefore, for every $q\in \tGa$ we have $p+\iota(p)-q\simeq q'$. Applying Dhar's burning algorithm and chip-firing from the point $q$, the fire stops when reaching $p$ and $\iota(p)$ and the graph is not entirely burnt. Furthermore, since the valency of $p$ and  $\iota(p)$ is $2$, the graph can be decomposed into two parts,
meaning that $\{p,\iota(p)\}$ is a disconnecting pair;
                \item [$(\Rightarrow)$] If $p$ is fixed by $\tj$ then  $\iota(p)$ is fixed as well. Moreover, since the genus of $\tGa$ is at least $3$, the involution $\tj$ has at least $4$ distinct fixed points. By contradiction, suppose that $\{p,\iota(p)\}$ disconnects the graph in two parts $A$ and $B$, therefore since $2p$ has rank $1$, using Dhar's burning algorithm and the fact that $val(p)=2$,
                for every point $q\in A$ there exists a unique point $q'\in B$ such that $2p-q\simeq q'$ and $\tj(q)=q'$. This is absurd because there are no fixed points other than $p$ and $\iota(p)$.
            \end{itemize}
        \end{itemize}

                For the final part, 
since $\tGa$ has genus $2g-1$, its hyperelliptic involution $\tj$ has exactly $2g$ fixed points, while $j$ has $g+1$ such points. We note that if $\tj(p)=p$, then also  $\tj(\iota(p))=\iota(p)$, therefore  the $2g$ fixed points of $\tj$ map down to $g$ fixed points for $j$. This means that there exists a unique  point $x\in \Gamma$ that is fixed by $j$ such that $\tj(p)=\iota(p)$. For instance, in \cref{abel1} the point $x$ is precisely the image in $\Ga$ under $\pi$ of the mid points of the crossed edges.
    \end{proof}
\end{proposition}

\begin{remark}\label{rem:fixedpoints}
    In the case where $\Gamma$ is not 2-edge-connected then, rather than  a unique fixed point of $j$  whose preimages are not fixed, we have a unique subtree of $\Gamma$. Indeed, this can be seen by contracting all the bridges of $\Gamma$ and applying \cref{fixedpoints}. We give this tree a name. 
    \end{remark}

\begin{definition}\label{def:weaklyFixed}
    Let $\pi:\wt\Gamma\to\Gamma$ be a free double cover of hyperelliptic metric graphs and let $T_f$ be the unique maximal fixed tree of $j$ whose preimages are not fixed by $\tj$. We call $T_f$ the \emph{weakly-fixed subtree} (or weakly-fixed point if $T_f$ is a single point) of the involution.

\end{definition}

\subsection{Construction $\star$}
We now present a useful construction, that we call Construction $\star$, for free double covers of  hyperelliptic metric graphs. The construction is similar 
 to \cite[Construction A]{LZ22}, with slight variations that take advantage of the hyperelliptic structure of the graph. 
As we shall see,  \emph{every} double cover of a hyperelliptic graph can be obtained via the construction.  
Note that we are not claiming that all the double covers obtained from Construction $\star$ are hyperelliptic, but we will see in \cref{constrhypcovers}
how to identify the ones that are. 
We begin by describing the case where $\Ga$ is 2-connected and then extend to the general case.

\subsection{The $2$-connected case}\label{sec:2edgeconnectedcase}
Let $\Gamma$ be  a 2-connected hyperelliptic metric graph of genus $g\geq 1$  with involution $j$.
 As we saw in  \cref{sec:fixedpoints}, the quotient $\Ga/j$ is a tree with $g+1$ leaves corresponding to the fixed points of the involution. We may, therefore, express $\Gamma$ as a union of two isomorphic trees 
 $T$ and $T'$ attached at their 1-valent vertices  $v_0,\dots,v_{g}$ and   $v'_0,\dots,v'_{g}$ respectively. Those attachment points correspond exactly to the fixed points of $j$. 
 
Now, any double cover of $\Gamma$ must contain two isomorphic copies of $T$ and two isomorphic copies of $T'$, which we denote  $\wt T^{\pm}$ and $\wt T'^{\pm}$ respectively. 

Since the graph $\Gamma$ is obtained by attaching each  $v_i$ to $v_i'$, the graph $\wt\Gamma$ is obtained by attaching each lift $\wt v_i^+$ of $v_i$ to either of the lifts $\wt v_i'^+$ or  $\wt v_i'^-$ of $v_i'$, and vice versa for the other lift $\wt v_i^-$ of $v_i$. These choices of gluing  are exactly what distinguishes the various double covers of $\Gamma$.

As there are $g+1$ fixed points $v_i$, we obtain $2^{g+1}$ double covers. However, there is some repetition here: as we will see in \cref{analogous},
swapping the gluing for all the vertices $\wt v_i'^+$ at the same time results in the same double cover, leaving us with $2^g$ distinct double covers.  
 See Figure \ref{fig1} for an example of the construction. 
 
For later use, denote 
\begin{equation}
   F=\{v_0,\ldots, v_g\}
\end{equation}
the set of vertices fixed by $j$ and by $S$ its subset whose preimage in $\wt\Gamma$ traverses either between $\wt T^{+}$ and $\wt T'^{-}$ or between $\wt T^{-}$ and $\wt T'^{+}$. Note that, since $\wt\Gamma$ is assumed to be connected, we can assume that $S$ is neither $\emptyset$ nor $F$.

\begin{example}
In \cref{fig1}, we consider a hyperelliptic metric graph $\Ga$ of genus $g\geq 2$ and draw the two trees corresponding to the quotient given by the involution whose fixed points are $v_0,\dots,v_g$. 

\begin{figure}[ht]
\centering
\begin{tikzpicture}
       \coordinate (1) at (-3,0.5);
	\coordinate (2) at (-2,0.5);
	\coordinate (3) at (-1.5,0.5);
	\coordinate (4) at (-0.5,0.5);
         \coordinate (5) at (0.5,0.5);
        \coordinate (6) at (1.5,0.5);
        \coordinate (7) at (2,0.5);
	\coordinate (8) at (3,0.5);
	\coordinate (9) at (-2.5,1);
	\coordinate (10) at (-1,1);
	\coordinate (11) at (1,1);
        \coordinate (12) at (2.5,1);
        \coordinate (13) at (-1.75,1.5);
	\coordinate (14) at (1.75,1.5);
        \coordinate (15) at (0,2.5); 
        \coordinate (1i) at (-3,0.5);
	\coordinate (2i) at (-2,0.5);
	\coordinate (3i) at (-1.5,0.5);
	\coordinate (4i) at (-0.5,0.5);
         \coordinate (5i) at (0.5,0.5);
        \coordinate (6i) at (1.5,0.5);
        \coordinate (7i) at (2,0.5);
	\coordinate (8i) at (3,0.5);
	\coordinate (9i) at (-2.5,-1+1);
	\coordinate (10i) at (-1,0);
	\coordinate (11i) at (1,0);
        \coordinate (12i) at (2.5,0);
        \coordinate (13i) at (-1.75,-0.5);
	\coordinate (14i) at (1.75,-0.5);
        \coordinate (15i) at (0,-1.5); 
        \coordinate (a1) at (-2.5,9.5);
	\coordinate (a2) at (-2,0.5+6);
	\coordinate (a3) at (-1.5,0.5+6);
	\coordinate (a4) at (-0.5,0.5+6);
         \coordinate (a5) at (0.5,0.5+6);
        \coordinate (a6) at (1.5,0.5+6);
        \coordinate (a7) at (2,0.5+6);
	\coordinate (a8) at (3,0.5+6);
	\coordinate (a9) at (-2.5,1+6);
	\coordinate (a10) at (-1,1+6);
	\coordinate (a11) at (1,1+6);
        \coordinate (a12) at (2.5,1+6);
        \coordinate (a13) at (-1.75,1.5+6);
	\coordinate (a14) at (1.75,1.5+6);
        \coordinate (a15) at (0,2.5+6); 
	\coordinate (a2i) at (-2,-0.5+6+1);
	\coordinate (a3i) at (-1.5,-0.5+6+1);
	\coordinate (a4i) at (-0.5,-0.5+6+1);
         \coordinate (a5i) at (0.5,-0.5+6+1);
        \coordinate (a6i) at (1.5,-0.5+6+1);
        \coordinate (a7i) at (2,-0.5+6+1);
	\coordinate (a8i) at (3,-0.5+6+1);
	\coordinate (a9i) at (-2.5,-1+6+1);
	\coordinate (a10i) at (-1,-1+6+1);
	\coordinate (a11i) at (1,-1+6+1);
        \coordinate (a12i) at (2.5,6);
        \coordinate (a13i) at (-1.75,-1.5+6+1);
	\coordinate (a14i) at (1.75,-1.5+6+1);
        \coordinate (a15i) at (0,-2.5+6+1); 
            \coordinate (b1) at (-3.2,9.5);
	\coordinate (b2) at (-2,0.5+12);
	\coordinate (b3) at (-1.5,0.5+12);
	\coordinate (b4) at (-0.5,0.5+12);
         \coordinate (b5) at (0.5,0.5+12);
        \coordinate (b6) at (1.5,0.5+12);
        \coordinate (b7) at (2,0.5+12);
	\coordinate (b8) at (3,0.5+12);
	\coordinate (b9) at (-2.5,13);
	\coordinate (b10) at (-1,13);
	\coordinate (b11) at (1,13);
        \coordinate (b12) at (2.5,1+12);
       \coordinate (b13) at (-1.75,1.5+12);
	\coordinate (b14) at (1.75,1.5+12);
        \coordinate (b15) at (0,2.5+12); 
	\coordinate (b2i) at (-2,-0.5+13);
\coordinate (b3i) at (-1.5,-0.5+13);
\coordinate (b4i) at (-0.5,-0.5+13);
        \coordinate (b5i) at (0.5,-0.5+13);
        \coordinate (b6i) at (1.5,-0.5+13);
        \coordinate (b7i) at (2,-0.5+13);
	\coordinate (b8i) at (3,-0.5+13);
	\coordinate (b9i) at (-2.5,12);
	\coordinate (b10i) at (-1,12);
	\coordinate (b11i) at (1,12);
        \coordinate (b12i) at (2.5,12);
    \coordinate (b13i) at (-1.75,-1.5+13);
    \coordinate (b14i) at (1.75,-1.5+13);
        \coordinate (b15i) at (0,-2.5+13); 
 \foreach \i in {1,2,3,4,5,6,7,8,9,10,11,12,13,14,15,1i,2i,3i,4i,5i,6i,7i,8i,9i,10i,11i,12i,13i,14i,15i,a1,a2,a3,a4,a5,a6,a7,a8,a9,a10,a11,a12,a13,a14,a15,a2i,a3i,a4i,a5i,a6i,a7i,a8i,a9i,a10i,a11i,a12i,a13i,a14i,a15i,b1,b2,b3,b4,b5,b6,b7,b8,b9,b10,b11,b12,b13,b14,b15,b2i,b3i,b4i,b5i,b6i,b7i,b8i,b9i,b10i,b11i,b12i,b13i,b14i,b15i,b2i,b3i,b4i,b5i,b6i,b7i,b8i,b9i,b10i,b11i,b12i,b13i,b14i,b15i}
	\draw[fill=black](\i) circle (0.15em);

\draw[thin, black] (1)--(9);
\draw[thin, black] (13)--(9);
\draw[thin, black] (2)--(9);
\draw[thin, black] (10)--(3);
\draw[thin, black] (10)--(4);
\draw[thin, black] (10)--(13);
\draw[thin, black] (5)--(11);
\draw[thin, black] (11)--(6);
\draw[thin, black] (12)--(7);
\draw[thin, black] (12)--(8);
\draw[thin, black] (14)--(11);
\draw[thin, black] (14)--(12);
\draw[thin, black] (14)--(15);
\draw[thin, black] (13)--(15);
\draw[thin, black] (1i)--(9i);
\draw[thin, black] (13i)--(9i);
\draw[thin, black] (2i)--(9i);
\draw[thin, black] (10i)--(3i);
\draw[thin, black] (10i)--(4i);
\draw[thin, black] (10i)--(13i);
\draw[thin, black] (5i)--(11i);
\draw[thin, black] (11i)--(6i);
\draw[thin, black] (12i)--(7i);
\draw[thin, black] (12i)--(8i);
\draw[thin, black] (14i)--(11i);
\draw[thin, black] (14i)--(12i);
\draw[thin, black] (14i)--(15i);
\draw[thin, black] (13i)--(15i);
\draw[thin, red] (1)--(1i);
\draw[thin, red] (2)--(2i);
\draw[thin, red] (3)--(3i);
\draw[thin, red] (4)--(4i);
\draw[thin, red] (5)--(5i);
\draw[thin, red] (6)--(6i);
\draw[thin, red] (7)--(7i);
\draw[thin, red] (8)--(8i);

\draw[thin, black] (a1)--(a9);
\draw[thin, black] (a13)--(a9);
\draw[thin, black] (a2)--(a9);
\draw[thin, black] (a10)--(a3);
\draw[thin, black] (a10)--(a4);
\draw[thin, black] (a10)--(a13);
\draw[thin, black] (a5)--(a11);
\draw[thin, black] (a11)--(a6);
\draw[thin, black] (a12)--(a7);
\draw[thin, black] (a12)--(a8);
\draw[thin, black] (a14)--(a11);
\draw[thin, black] (a14)--(a12);
\draw[thin, black] (a14)--(a15);
\draw[thin, black] (a13)--(a15);
\draw[thin, black] (a13i)--(a9i);
\draw[thin, black] (a2i)--(a9i);

\draw[thin, black] (a10i)--(a3i);
\draw[thin, black] (a10i)--(a4i);
\draw[thin, black] (a10i)--(a13i);
\draw[thin, black] (a5i)--(a11i);
\draw[thin, black] (a11i)--(a6i);
\draw[thin, black] (a12i)--(a7i);
\draw[thin, black] (a12i)--(a8i);
\draw[thin, black] (a14i)--(a11i);
\draw[thin, black] (a14i)--(a12i);
\draw[thin, black] (a14i)--(a15i);
\draw[thin, black] (a13i)--(a15i);
\draw[thin, red] (a2)--(a2i);
\draw[thin, red] (a3)--(a3i);
\draw[thin, red] (a4)--(a4i);
\draw[thin, red] (a5)--(a5i);
\draw[thin, red] (a6)--(a6i);
\draw[thin, red] (a7)--(a7i);
\draw[thin, red] (a8)--(a8i);

\draw[thin, black] (b13)--(b9);
\draw[thin, black] (b2)--(b9);
\draw[thin, black] (b10)--(b3);
\draw[thin, black] (b10)--(b4);
\draw[thin, black] (b10)--(b13);
\draw[thin, black] (b5)--(b11);
\draw[thin, black] (b11)--(b6);
\draw[thin, black] (b12)--(b7);
\draw[thin, black] (b12)--(b8);
\draw[thin, black] (b14)--(b11);
\draw[thin, black] (b14)--(b12);
\draw[thin, black] (b14)--(b15);
\draw[thin, black] (b13)--(b15);

\draw[thin, black] (b13i)--(b9i);
\draw[thin, black] (b2i)--(b9i);
\draw[thin, black] (b10i)--(b3i);
\draw[thin, black] (b10i)--(b4i);
\draw[thin, black] (b10i)--(b13i);
\draw[thin, black] (b5i)--(b11i);
\draw[thin, black] (b11i)--(b6i);
\draw[thin, black] (b12i)--(b7i);
\draw[thin, black] (b12i)--(b8i);
\draw[thin, black] (b14i)--(b11i);
\draw[thin, black] (b14i)--(b12i);
\draw[thin, black] (b14i)--(b15i);
\draw[thin, black] (b13i)--(b15i);

\draw[thin, black] (a1)--(b9i);
\draw[thin, red] (b3)--(b3i);
\draw[thin, red] (b4)--(b4i);
\draw[thin, red] (b5)--(b5i);
\draw[thin, red] (b6)--(b6i);
\draw[thin, red] (b7)--(b7i);
\draw[thin, red] (b8)--(b8i);

  \draw (8) node [right ] {$v_{g}$};
  \draw (1) node [right] {$v_0$};

    \draw (a1) node [right] {$\widetilde{v}^-_0$};   
     \draw (b1) node [right] {$\widetilde{v}^+_0$};
    \draw (a8) node [right ] {$\widetilde v^-_{g}$};
    \draw (b8) node [right ] {$\widetilde v^+_{g}$};

 \draw (0,0.4) node [above] {$\dots$};
  \draw (0,-0.4) node [below] {$\dots$};
\draw (0,1.4) node [above] {$\dots$};
  \draw (0,0.4+6) node [above] {$\dots$};
  \draw (0,-0.4+6) node [below] {$\dots$};
\draw (0,1.4+6) node [above] {$\dots$};
\draw (0,0.4+6+6) node [above] {$\dots$};
  \draw (0,-0.4+6+6) node [below] {$\dots$};
\draw (0,13.4) node [above] {$\dots$};
   \draw (0,3.2) node [above] {$\downarrow$};
   \draw[thin, black] (b9) to [out=180+65, in=90] (b1);
   \draw[thin, black] (b1) to [out=180+90, in=90+25] (a9i);
   \draw (2.4,8) node [above] {$\widetilde T^-$};
   \draw (2.4,2) node [above] {$T$};
   \draw (2.4,14) node [above] {$\widetilde T^+$};
 \draw (2.4,-1.5) node [above] {$T'$};
 \draw (2.4,4.5) node [above] {$\widetilde T^{'-}$};
 \draw (2.4,10.5) node [above] {$\widetilde T^{'+}$};
   \end{tikzpicture}
   
 \caption[below]{The Construction $\star$ of the double cover when $S=\{v_0\}$. }
  \label{fig1}
 \end{figure}
\end{example}

\subsection{The general case with cut vertices and bridges}\label{generalcase}

Let $\Gamma$ be a hyperelliptic graph of genus $g\geq 2$.
Denote by $\Gamma_0,\dots,\Gamma_{n}$ the $2$-connected components of genera $g_0, g_1,\ldots,g_n$ respectively. Note that every cut vertex of $\Gamma$  corresponds to distinct vertices in the various 2-connected components that contain it. 

As the hyperelliptic involution fixes  any bridge or cut vertex point-wise, each subgraph $\Ga_i$ is hyperelliptic as well with $g_i+1$ fixed points. 
Therefore,  the total number of  fixed points on the \emph{disjoint} union of the $\Gamma_i$s is: \[\sum_{i=0}^{n}(g_{i}+1)=g+1+n.\]

Applying Construction $\star$ to each connected component, we can write 
 $\Gamma_i = T_i\cup T_i'$ for each $i$, where  $T_i$ and $T_i'$ are isomorphic trees, and $\wt\Gamma_i = \wt T_i^{\pm}\cup\wt T_i'^{\pm}$.  
 For each $i$, denote $F(\Gamma_i)$  the set of vertices  connecting the trees $T_i$ and $T'_i$.
 As before, for each $F(\Gamma_i)$ we denote $S_i$ the set of vertices  traversing between $\wt T_i^{\pm}$ and $\wt T_i'^{\mp}$.
Now take  $S=\sqcup_i S_i$ and  $F=\sqcup_i\, F(\Ga_i)$ . Note that some vertices of $\Gamma$ may appear in more than one $\Gamma_i$ and we  distinguish them in $S$ and in $F$. 
  
 We are left with lifting  the bridges of $\Gamma$.  Different lifts of a bridge result in an isomorphic double cover, so we might as well assume that, when a bridge $b$ connects vertices $u$ and $v$ of $\Gamma$ whose preimages in $\wt\Gamma$ are $u^{\pm},v^{\pm}$,
 one of its preimages  $\wt b^{+}$ connects  $\tilde u^+$ with $\tilde v^+$ and the other preimage $\wt b^{-}$ connects  $\tilde u^-$ with $\tilde v^-$.

In conclusion, since there are $g+1+n$ fixed points, we end up with  $2^{g+1+n}$ double covers. As before, there is repetition and some of those are isomorphic double covers.

\begin{example}
\cref{fig:constructionbridges} shows 
    the construction of a double cover of a
    hyperelliptic metric graph with bridges. We first use the Construction $\star$ on each 2-edge-connected component and then lift the bridges. For visualisation, we use red segments to indicate which vertices are attached to each other, they can be thought of as edges of length 0. 
    In this example we choose three pairs of edges connecting $\widetilde{T}^+$ and $\widetilde{T}^{'-}$. Note that the cover graph is not hyperelliptic, although we started with a hyperelliptic one.
\begin{figure}[ht]
    \centering
\input{pic2}
 \caption{ The construction in case of bridges where $|S|=3$.}
    \label{fig:constructionbridges}
\end{figure}

\end{example}

\subsection{Counting the number of distinct free double covers of hyperelliptic graphs}
From the total number of $2^{g+1+n}$ free double covers of a hyperelliptic graph $\Ga$ with $n$ bridges constructed above, some  are isomorphic and many are not hyperelliptic. 

Let $I:=\{0,\dots,n\}$ be the labelling of the $2$-connected  components of $\Gamma$ (namely, the components that remain after removing bridges and separating  cut-vertices)  and let $S=\sqcup_{i\in I} S_i$, where, as before, $S_i$ is the subset of vertices swapped between $\widetilde{T}_i^{\pm}$ and $\widetilde{T}_i^{'\mp}$ (where vertices appearing in two components are counted twice). We denote by $S^c_i$ the complement of the vertices of $S_i$ in the connected component $\Ga_i$. 
For a subset $J$ of $I$, we denote $S^c_J$ the set obtained from $S$ by swapping each $S_j$ for  $j\in J$ with its complement, namely $S^c_J = \sqcup_{j\in J}S^c_j\sqcup_{j\in I\setminus J}S_j$.

\begin{proposition}\label{analogous}
    The free double covers associated to
      $S$ and $S^c_J$ are isomorphic for any subset $J\subset I$. Conversely, if the double covers associated with $S$ and $S'$ are isomorphic then  $S' = S_J^c$, for some  subset  $J\subset I$.

    \begin{proof}
    
        Let $\pi_{S}\colon \tGa_S\to \Ga$ and $\pi_{S_J^c}\colon \tGa_{S_J^c}\to \Ga$ be the free double covers associated to $S$ and $S_J^c$ respectively. 
        We wish to construct an isomorphism $\alpha\colon  \tGa_S\to \tGa_{S_J^c}$ such that the following diagram commutes: 
        \[
\begin{tikzcd}
\tGa_S \arrow[rr, "\alpha"] \arrow[rd, "\pi_S"] &      & \tGa_{S_J^c} \arrow[ld, "\pi_{S_J^c}"] \\
                              & \Ga&                
\end{tikzcd}.
        \]

 Let $\Gamma_i$ be a 2-connected component of $\Gamma$ such that $i\in J$ and,  as usual, let $T_i,T_i'$ be the two  isomorphic trees whose union is $\Gamma_i$. Denote by $\wt\Gamma_{i,S}$ and $\wt\Gamma_{i,S_J^c}$ the subgraphs of $\wt\Gamma_S$ and $\wt\Gamma_{S_J^c}$ respectively corresponding to $\Gamma_i$, namely
 $\wt\Gamma_{i,S} = \pi_{S}^{-1}(\Gamma_i), \wt\Gamma_{i,S_J^c} = \pi_{S_J^c}^{-1}(\Gamma_i)$.

        By a slight abuse of notation, we use the  symbols $\wt T_i^{\pm}$ and $\wt T_i'^{\pm}$ to refer to the preimages of the trees $T_i, T_i'$
        in both
        $\wt\Gamma_{i,S}$ and $\wt\Gamma_{i,S_J^c}$ and use the same notation for the fixed and swapped vertices.  
    
        We define $\alpha_i:\wt\Gamma_{i,S}\to\wt\Gamma_{i,S_J^c}$  to be the morphism of graphs that maps the trees $\wt T_i^+$ and $\wt T_i'^+$ of $\wt\Gamma_{i,S}$ to the same trees $\wt T_i^+$ and $\wt T_i'^+$ of $\wt\Gamma_{i,S_J^c}$, but 
        swaps $\wt T_i^-$ and $\wt T_i'^-$.
        Since a vertex $\wt v_i^+$ of $T_i^+$ is glued to $\wt v_i^+$ in $\wt\Gamma_{i,S}$ if and only if 
        the corresponding vertex  is glued to $\wt v_i^-$ in  $\tGa_{S_J^c}$, this map is continuous. 
The map $\alpha$ on the entire graph is now obtained      by repeating this construction for each $i\in J$ and by attaching the preimages of the bridges accordingly.

Conversely, suppose that the double covers $\pi:\wt\Gamma\to\Gamma$ and $\pi':\wt\Gamma'\to\Gamma'$ corresponding to $S$ and $S'$ are isomorphic.
 
We may assume that both $S$ and $S'$ contain some vertex $v_0$, otherwise replace $S$ with its complement.  Now,  every vertex $v\in F$ is  in $S$ if and only if there is a simple cycle in $\wt\Gamma$ that passes through both  preimages of $v$ and the preimages of $v_0$.  An analogous property holds for $\pi'$. In particular, a vertex $v$ is in $S$ if and only if it is in $S'$. 
Therefore, there exists $J\subset I$ such that $S'=S^c_{J}$.

    \end{proof}
\end{proposition}

\begin{remark}
     From \cref{analogous},  it follows that the number of  free double covers up to isomorphism obtained from Construction $\star$ is $2^g$. On the other hand, it is known that the total number of double covers of a graph of genus $g$ is also $2^g$ (see for instance \cite[Construction A,B]{LZ22}, \cite[Section 5]{JensenLen_thetachars}, or \cite{Waller_DoubleCoversOfGraphs}). It follows that every double cover of a hyperelliptic graph can obtained via  Construction $\star$.
\end{remark}

As mentioned above, not all of the double covers produced via Construction $\star$  are hyperelliptic. 
The following proposition gives a sufficient and necessary condition for $\tGa$ to be hyperelliptic as well.

\begin{proposition}\label{constrhypcovers}
     Let $\pi\colon \tGa\to \Gamma$ be a free double cover of metric graphs such that $\Gamma$ has  genus $g_{\Ga}\geq 2$ and no points of valency $1$. 
Then $\tGa$ is hyperelliptic if and only if there is a tree $T_f\subseteq\Gamma$ (possibly consisting of a single point) fixed by $j$ such that $S$ or $S^c_J$ consists precisely of the vertices of $F$ that are also leaves of $T_f$.

     \begin{proof} 

Suppose that $\wt\Gamma$ is hyperelliptic with involution $\tj$. By   \cref{rem:fixedpoints}, there is a unique tree $T_f\subset\Gamma$ fixed by $j$ whose preimages are swapped by $\tj$. Let   $S'$ be the subset of $F$ consisting  of the leaves of $T_f$. Then it is straightforward to check that  
the double cover obtained by applying Construction $\star$ to $S'$ is isomorphic to $\pi$. Using \cref{analogous}, we may assume that $S=S'$. 

Conversely, suppose that $S$ consists of the vertices of $F$ adjacent to the tree $T_f$. Define an involution $\tj$ on $\wt\Gamma$ whose restriction to each of the trees $T_i^{\pm}$ coincides with the corresponding involution of $\Gamma$, fixes the vertices in the complement of $S$ and swaps the vertices in $S$. Then it is straightforward to check that $\tj$ is a hyperelliptic involution.

     \end{proof}
\end{proposition}

As an immediate corollary, we can count the number of free double covers of a hyperelliptic graph. 

\begin{corollary}\label{cor:numberOfHyperelliptic}
   Let $\Gamma$ be a hyperelliptic metric graph. Then, up to isomorphism, it has $g+1$ free double covers by a hyperelliptic metric graph.
\end{corollary}
}

\begin{proof}
       The number of covers by hyperelliptic graphs doesn't change when contracting bridges, so we may as well assume that $\Gamma$ is
     $2$-edge connected and that the weakly-fixed tree is , in fact, a point $x$.

       It follows from \cref{analogous} and \cref{constrhypcovers}, that for every choice of a $j$-fixed point $x$, we can construct exactly one hyperelliptic double cover up to isomorphism. Therefore,
        the number of distinct hyperelliptic double covers is equal to the number of fixed points of $\Ga$, which is $g+1$.
        
\end{proof}
\begin{remark}
    For algebraic étale double covers $\widetilde C\to C$ of smooth curves, the number of covers by hyperelliptic curves is ${2g+2\choose 2}$ among the $2^{2g}-1$ distinct connected covers \cite{HMfarkashypcovers}. 
   Such covers are indexed by choosing $2$  of the $2g+2$ fixed points of the hyperelliptic involution on $C$ and letting them not be fixed by the hyperelliptic involution on $\wt C$. 
     Moreover, by \cite{Panizzut} we know that the  fixed points of the hyperelliptic involution on $C$ tropicalise in pairs to the $g+1$ fixed points of the tropical hyperelliptic involution, therefore it makes sense to have this result in the tropical setting as well.
  
\end{remark}

We end this section by leveraging the fact that covers of hyperelliptic graphs are so constrained    to provide a simple description for  the $\Psi$-collapsible locus in that case. 

\begin{lemma}\label{lem:CollapsibleLocusOfHyperelliptic}
    Suppose that $\pi:\wt\Gamma\to\Gamma$ is a free double cover of hyperelliptic metric graphs. If the weakly-fixed tree is a non-disconnecting point in the middle of an edge $e$, then the cyclic $\Psi$-collapsible locus consists of a single properly disconnecting pair covering $e$. In all other cases, the locus is empty. 
\end{lemma}

\begin{proof}

    We may assume that $\wt\Gamma$ is 2-edge-connected since contracting the bridges has no effect on the cyclic $\Psi$-collapsible locus. We may also assume that the weakly-fixed tree is a point $x$ since its preimages in $\wt\Gamma$ are never in the $\Psi$-collapsible locus. Denote $x',x''$ the preimages of $x$ in $\wt\Gamma$. 

Let $f$ be an edge of $\wt\Gamma$ not adjacent to $x',x''$.
We may assume, without loss of generality, that $f$ belongs to $\wt T^+$.
We claim that, if  none of the vertices of $f$ are in $S$ (that is, $f$ does not traverse between $\wt T^+$ and $\wt T^-$), then removing $f$ and $\iota f$ does not disconnect $\wt\Gamma$. To see that, start from a path $\gamma$ in 
$\Gamma\setminus\pi(f)$ between the endpoints of $\pi(f')$. Since $f$ is not adjacent to $S$, it follows that $\gamma$ lifts to a path 
between the endpoints of $f$ that is fully contained in $\wt T^+$ and, in particular, doesn't go through $\iota f$. Similarly, there is a path in $\wt\Gamma$ between the endpoints of $\iota f$ that doesn't go through $f$. It follows that  $\wt\Gamma\setminus\{f',f''\}$ is connected.

By the assumption that $\wt\Gamma$ is hyperelliptic, \cref{constrhypcovers} implies that the  vertices in $S$ are all adjacent to the weakly-fixed point. We conclude that the only edges that could be contracted by $\Psi$ are the ones adjacent to $x'$ or $x''$.

Suppose now that the weakly-fixed point $x$ is disconnecting and let $f'$ be an adjacent edge. Then the removal of $\{f',\iota f'\}$ does not disconnect the graph, since  there is a path from $x'$ to $x''$ via the complement of the 2-edge-connected component of $\wt\Gamma$ containing $f'$. 

Conversely, if the weakly-fixed point $x$ is non-disconnecting then, by \cref{fixedpoints}, its preimages $x',x''$ are  a disconnecting pair, so the edges containing them are contracted by $\Psi$.

\end{proof}

\section{The image of the Abel--Prym map for hyperelliptic double covers}
\label{sec:hypcase}

In this final section, we continue our investigation of the  Abel--Prym graph, under the assumption that the source graph $\wt\Gamma$ is hyperelliptic. We show that its Jacobian
is isomorphic, as a principally polarised abelian variety, to the Prym variety of the double cover. 
We begin by providing two different descriptions of the image, one that is simple and abstract and one that is more explicit.

 \begin{lemma}\label{lem:quotientbyinvolution}
Let $\wt\Gamma$ be a $2$-edge connected  hyperelliptic metric graph. Then the Abel--Prym graph is isomorphic to the quotient graph $\tGa/{(\iota\circ \tj)}$. In particular, the Abel--Prym graph is hyperelliptic.
     \end{lemma}

Here, we are taking the quotient in the sense of (see \cite[Section 2.2]{CH} or \cite[Section 1.3]{LenUlirschZakharov_abelianCovers}) with respect to the canonical model of $\wt\Gamma$. In particular, if a segment is flipped by the Abel--Prym map, then it will get contracted in the quotient.

  \begin{proof}  \label{graphModuloIJ}

The map $\iota\circ \tj$ is indeed a harmonic automorphism of degree $2$, since 
  both $\iota$ and $\tj$ are isometric involutions. 
    Set-theoretically,  the Abel--Prym map identifies every point $p$ with $\iota\circ\tj(p)$. So we need to check that it contracts the same edges as  the quotient map  and prescribes the same edge-lengths.

From the definition of the quotient,  
every edge gets dilated by a factor of its stabilizer and edges are contracted when they are flipped by the involution. So an edge $e$ is dilated by a factor of 2 whenever it is fixed by $\iota\circ\tj$, namely, when $\iota(p) = \tj(p)$ for all $p\in e$. But then $j(\pi(p)) = \pi(p)$, so $\pi(p)$ is fixed by the hyperelliptic involution. Using \cref{fixedpoints} and the fact that the graph is 2-edge-connected, we see that $e$ and $\iota e$ are a properly disconnecting pair such that $\pi(e)$ is a bridge, so the edge $e$ is also dilated by the Abel--Prym map.

  Now let $f$ be an edge flipped by $\iota\circ\tj$. Then the middle point $p$ of $f$ is fixed, meaning that $\tj(p)=\iota(p)$. The point $\pi(p)$ corresponds to the weakly-fixed point  
 and  by \cref{fixedpoints}, the set $\{p,\iota(p)\}$ is a disconnecting pair. It follows that, the edge $f$ containing $p$  gets contracted by $\Psi$. Conversely, if $\Psi$ contracts edges $\{f,\iota(f)\}$ which are not bridges then, by hypothesis, they form a disconnecting pair and by the hyperellipticity of $\tGa$, we have $\tj(f)=\iota(f)$.

         Finally, $\Psi(\wt\Gamma)$ is hyperelliptic as the quotient of a hyperelliptic graph by a harmonic automorphism.

\end{proof}

\begin{remark}
The lemma doesn't hold when $\wt\Gamma$ is not assumed to be 2-edge-connected. In that case, the bridges get contracted by the Abel--Prym map but are not flipped by $\iota\circ\tj$ (regardless of the choice of model).  
Even when $\wt\Gamma$ is 2-edge-connected, the quotient $\wt\Gamma/\iota\circ\tj$ doesn't always coincide with the topological quotient since the former contracts the cyclic $\Psi$-collapsible locus while the latter doesn't. 
\end{remark}

\subsection{The graph $\Gamma^{\dagger}$}\label{GammaDagger}
Given a free double cover $\pi:\wt\Gamma\to\Gamma$ of hyperelliptic graphs, we now describe a graph $\Gamma^{\dagger}$ that, as we will show in \cref{structureOfImage}, is isomorphic as a metric graph to the image of the Abel--Prym map. 
Let $T_f$ be the weakly-fixed subtree of $j$. 

If removing $T_f$ does not disconnect $\Gamma$, then $T_f$ is a single point $x_f$ of valency 2 and the preimage in $\wt\Gamma$ of every bridge of $\Gamma$ is a pair of bridges.
In this case,   $\Gamma^{\dagger}$ is the graph obtained from $\Gamma$ by deleting the two edges adjacent to $x_f$ and contracting all the bridges. 

If removing $T_f$ does disconnect $\Gamma$ then the preimage in $\wt\Gamma$ of every edge of $T_f$ is a properly disconnecting pair, and the preimage of every other bridge of $\Gamma$ is a pair of bridges. In this case, 
$\Gamma^{\dagger}$ is obtained from $\Gamma$ by replacing $T_f$ with two isomorphic copies of $T_f$ stretched by a factor of 2, and, for any other pair of edges $e,j(e)$ of $\Gamma$ terminating at $T_f$, letting one of them terminate at one copy and the other terminate at the other copy (the choice of copy will not affect the isomorphism type of the graph). All the other bridges are contracted.

\begin{proposition} \label{structureOfImage}

 Let $\pi\colon \tGa\to \Gamma$ be a free double cover of hyperelliptic metric graphs of genus $g_{\Gamma}\geq 2$ such that $\tGa$ is 2-edge-connected. Then $\Psi(\wt\Gamma) = \Gamma^{\dagger}$, where $\Gamma^{\dagger}$ is described in \cref{GammaDagger}.
In particular, the Abel--Prym graph has genus $g_{\Ga}-1$.

    \begin{proof}

The statement about the genus is straightforward from the first part of the statement, since the graph $\Gamma^{\dagger}$ is obtained from $\Gamma$ by either keeping the same number of vertices and removing $|E(T_f)|$ edges or by keeping the number of edges but adding $|V(T_f)|$ vertices. 
So we need to prove that $\Psi(\wt\Gamma) = \Gamma^{\dagger}$. For simplicity of notations, we will assume that $\Gamma$ is 2-edge-connected. The general case requires only minor adjustments.

Suppose first that $x$ is not a cut vertex. By a slight abuse of notation, we will identify points of $\Gamma^{\dagger}$ with points in $\Gamma$.
Recall that $\Gamma$ consists of two copies $T_1$ and $T_2$ of a tree $T$ glued at the fixed point of $j$ and that $\wt\Gamma$ can be identified with four copies $T_1^+,T_1^-,T_2^+,T_2^-$ of $T$. If $x_1^+, x_1^-,x_2^+,x_2^-$ are the points corresponding to $x$ in the four trees, then  the trees $T_i^+,T_i^-$ (for $i=1,2$) are glued with each other  at all  their leaves, except for the preimages of the point $x$, in which $x_1^+$ is glued to $x_2^-$ and $x_2^+$ is glued to $x_1^-$.  

If $y\in\wt\Gamma$ is such that $\pi(y)$ doesn't belong to the edge containing $x$, the pair $y,z=\iota\circ\tj(y)$ does not disconnect $\wt\Gamma$ and, therefore, $z$ is the unique point other than $y$ such that 
$\Psi(y)=\Psi(z)$. 
If $y\in T^+_i$, then $z\in T^-_{i+1}$ (where the index is taken modulo 2) and vice versa. Furthermore, if $y$ and $z$ are in the interior of edges, the map $\Psi$ is a local isometry on each of them. 
If, on the other hand, $\pi(y)$ does belong to the edge containing $x$, then $y$ and $z$ do form a disconnecting pair, so their entire edges get contracted by $\Psi$. 
It follows that $\Psi(\wt\Gamma)$ contains an isometric copy of $T_1^+\cup T_2^+$, except for contracting  each of the four edges emanating from the two points $x_i^{\pm}$. But contraction of an edge amounts to deleting it and identifying its endpoints, so the result follows. 

Now, suppose that $x$ is a cut vertex and suppose that $\val(x) = 2k$. Then removing $x$ disconnects $\Gamma$ into $k$ connected components. Recall that, for each of them, we have a pair of trees $T_i^1,T_i^2$ and $\wt\Gamma$ contains  two copies  of each tree, denoted by $T_i^{1+}, T_i^{1-}, T_i^{2+},T_i^{2-}$ with points $x_i^{j\pm}$ corresponding to $x$ in each. To obtain $\wt\Gamma$, glue the leaves of   $T_i^{1+}$ with those of $T_i^{2+}$ and the leaves of $T_i^{1-}$ with those of $T_i^{2-}$ as in $\Gamma$, except that $x_i^{1+}$ is glued to $x_i^{2-}$ and $x_i^{1-}$ is glued to $x_i^{2+}$.
Similarly to the first case, $\Gamma^{\dagger}$ will contain an isometric copy of $T_i^{1+}\cup T_i^{2+}$ for each $i$, except for the edges emanating from $x$. However, unlike the first case, when $y$ is on an edge emanating from $x$, the points $y,\iota(\tj (y))$ do not form a disconnecting pair, so $\Psi(\wt\Gamma)$ contains an isometric copy of each of those edges.

    \end{proof}
\end{proposition}

\subsection{The Jacobian of the Abel--Prym graph}\label{sec:JacobianVSPrym}

Using the description above, we can relate the Jacobian of the Abel--Prym map with the Prym variety of the double cover. 

We quickly remind the reader about the structure of tropical abelian varieties and the metric induced by the Prym variety. See \cite[Section $2$]{LZ22} and \cite[Section $4$]{RZ_ngonal} for more details. Let $\Lambda$ and $\Lambda'$ be finitely generated free abelian groups of the same rank and let $[\cdot, \cdot] : \Lambda \times \Lambda' \to \R$ be a non-degenerate pairing. The triple $(\Lambda, \Lambda', |\cdot, \cdot])$ defines a \emph{real torus with integral structure} (or simply an \emph{integral torus})  
 \[
 \Si=\Hom(\Lambda, \R) / \Lambda',
 \]
 where the inclusion $\Lambda' \subseteq \Hom(\Lambda, \R)$ is given by $\lambda' \mapsto [\cdot,\lambda']$.
The \emph{dimension} of an integral torus is $\dim_{\R}\Si=\rk \La=\rk \La'$. 
 A \emph{polarisation} on $\Sigma$ is a group homomorphism $\xi : \Lambda' \to \Lambda$ such that 
 \[
 (\cdot,\cdot) = [\xi(\cdot), \cdot] \colon \Lambda'_\R \times \Lambda'_\R \to \R
 \]
 is a symmetric and positive definite bilinear form. A polarisation is necessarily injective, and is called \emph{principal} if it is also bijective.
  An integral torus together with a principal polarisation is called \emph{principally polarised tropical abelian variety} or \emph{pptav} for short.

A \emph{homomorphism of integral tori} $f=(f^{\#},f_{\#})\colon (\Lambda_1,\Lambda_1',[\cdot,\cdot]_1)\to (\Lambda_2,\Lambda_2',[\cdot,\cdot]_2)$ consists of a pair of homomorphisms $f^{\#}\colon \Lambda_2\to \Lambda_1$
and $f_{\#}\colon \Lambda_1'\to \Lambda_2'$
satisfying the relation 
\[
[f^{\#}(\lambda_2),\lambda_1']_1=[\lambda_2,f_{\#}(\lambda_1')]_2
\]
for all $\lambda_1'\in\Lambda_1'$ and $\lambda_2\in \Lambda_2$.
A homomorphism $f=(f^{\#},f_{\#})$ is an \emph{isomorphism} if $f_{\#}$ and $f^{\#}$
are isomorphisms.

The quintessential example of a pptav is the  \emph{Jacobian} of a metric graph $\Ga$, given by
\[
\Si_1=\mathrm{Jac}(\Ga)=(H^1(\Ga,\ZZ),H_1(\Ga,\ZZ),[\cdot,\cdot]_1),
\] 
where the pairing $[\cdot,\cdot]_1$ is given by integration along cycles and the principal polarisation is natural map identifying  $H_1(\Gamma,\ZZ)$ with $H^1(\Gamma,\ZZ)$. Note that some sources use $\Omega(\wt\Gamma,\ZZ)$ rather than $H^1(\wt\Gamma,\ZZ)$ but we don't require that perspective here.

Arguably, the second most important example of an abelian variety is the  Prym variety  associated to a free double cover $\pi\colon \tGa \to \Ga$,  given by
\[ 
\Sigma_2=\Prym(\tGa/\Ga)=((\text{Coker}\,\pi^*)^{tf},\mathrm{Ker}\,\pi_*,[\cdot,\cdot]_2).
\]
Here, $\pi^*$ and $\pi_*$ are the pullback and pushforward on cohomology and homology respectively, and $(\text{Coker}\,\pi^*)^{tf}$ is the quotient of $(\text{Coker}\,\pi^*)$ by its torsion subgroup. The most convenient way of thinking of $\mathrm{Ker}\pi_*$ is as antisymmetric cycles in $\wt\Gamma$, namely oriented cycles that flip sign under the action of the involution. The pairing $[\cdot,\cdot]_2$ is induced by the integration pairing on $\Jac(\tGa)$ and the principal polarisation is obtained from the principal polarisation on the Jacobian  divided by 2. The Prym variety is a pptav of dimension $g-1$.

\medskip

\subsubsection{Bases for homology}
Given a spanning tree $\calT$ for $\Gamma$, we now describe a
natural bases for $H_1(\Gamma,\ZZ), H_1(\wt\Gamma,\ZZ)$, and the antisymmetric cycles on $\wt\Gamma$, following  \cite[Construction B]{LZ22}). Throughout, fix an orientation on $\Gamma$ and a compatible orientation on $\wt\Gamma$.
As usual, we follow the convention that the preimage in $\wt\Gamma$ of an edge $e$ are denoted $\tilde e^+$ and $\tilde e^-$. 
Fix an edge $e_g$   of $\Gamma\setminus \calT$.  
 The preimage of $\calT$ in $\wt\Gamma$ has two connected components, denoted by 
$\wt{\calT}^{\pm}$, but adding $\wt e_g^+$ we obtain a spanning tree
\[
\widetilde{\calT}=\widetilde{\calT}^{+}\cup\widetilde{\calT}^{-}\cup \wt{e_g}^+.
\]
for $\wt\Gamma$.
For each of the edges $e_i$ in the complement of $\calT$, other than $e_g$, we define $\wt\gamma^+$ and $\wt\gamma^-$  the unique cycle 
of $\widetilde{\calT}\cup \{\widetilde{e}_{i}^{\pm}\}$, such that $\langle\tga_{i}^{\pm},\widetilde{e}_{i}^{\pm}\rangle=1$. Similarly, the cycle $\wt\gamma_g^-$ is the unique cycle of $\widetilde{\calT}\cup \{\widetilde{e}_{g}^{-}\}$ such that $\langle\tga_{g}^{-},\widetilde{e}_{g}^{-}\rangle=1$.   We will occasionally denote edges in the complement as $e$ (without specifying an index). In that  case, the corresponding cycles will be denoted by $\gamma_e$ and $\tilde\gamma_e^{\pm}$.

Now,
\begin{equation*}
    \{\tga^{\pm}_{1},\dots, \tga^{\pm}_{g-1}, \tga_g^-\}
\end{equation*}
is a $\ZZ$-basis for $H_1(\widetilde{\Gamma},\ZZ)$, and, more importantly for our purpose,  the set
\[
\calC = \left\{\wt\gamma_i^+ - \iota\wt\gamma_i^+\right\}_{i=1,\ldots,g-1}
\]
is a $\ZZ$-basis for the anti-symmetric cycles of $\wt\Gamma$, and hence for $\mathrm{Ker}\,\pi_*$. 

It will be useful to have a more detailed description for those cycles. Let $e\neq e_g$ be one of the  edges in the complement of $\calT$. If the lifts $\tilde e^{\pm}$ of $e$ do not traverse between the trees $\wt\calT^+$ and $\wt\calT^-$, then $\pi^{*}\gamma_e$ is the disjoint union of the cycles 
$\wt\gamma^+_e$ and $\iota\wt\gamma^+_e$. The multiplicity of each edge appearing in $\wt\gamma^+_e$ coincides with the multiplicity of the corresponding edge of $\gamma_e$ and has absolute value 1. 
Since $\wt\gamma^+_e$ is disjoint from $\iota\wt\gamma^+_e$, the multiplicity of each edge appearing in
$\wt\gamma^+_e - \iota\wt\gamma^+_e$ also has absolute value 1. 

When the preimages of $e$ do traverse between  the trees $\wt\calT^+$ and $\wt\calT^-$, the situation is a bit more complicated. Let $\tilde v^+$ and $\tilde v_g^+$ be the vertices of $\tilde e^+$ and $\tilde e_g^+$ respectively that belong to the tree $\wt\calT^+$ and let $v^+$ and $v_g^+$ be their images in $\Gamma$. Similarly, denote $u, u_g$ and $\tilde u, \tilde u_g$  the other end points of those edges. Then $\wt\gamma^+_e$ consists of the path along $\calT^+$ from $\tilde v_g^+$ to $\tilde v^+$, the edge $\tilde e^+$, the path in $\calT^-$ from $\tilde u$ to $\tilde u_g$, and finally the edge $e_g^+$. An edge appears with multiplicity 2 in  the antisymmetric cycle $\wt\gamma^+_e-\iota\wt\gamma^+_e$ when both it and its reflection appear in $\wt\gamma^+_e$ with opposite orientation. 
If we denote by $\delta_+$  the path from $v^+$ to $v_g^+$  and by $\delta_{-}$  the path from $u^-$ to $u_g^-$, then those are precisely the edges  in the preimage of $\delta_+\cap\delta_{-}$. 

We now specialise to the case where $\Gamma$ and $\wt\Gamma$ are both hyperelliptic  and choose the spanning tree $\calT$  with the property that, at each 2-connected component, the tree  includes exactly one of the two edges adjacent to the weakly-fixed tree.

Choose the edge $\tilde e_g$ that connects $\wt\calT^+$ with $\wt\calT^-$  to be adjacent to the weakly-fixed tree. Then an edge $e$ in the complement of $\calT$  lifts to edges that traverse between the $\calT^+$ and $\calT^-$ precisely when $e$ is adjacent to the weakly-fixed tree. For such an edge, the path $\delta_+$ from $v_g^+$ to $v^+$ is contained in the weakly-fixed tree, so  the intersection between $\delta_+$ and $\delta_-$ is contained in the weakly-fixed tree as well. It follows that the only edges of $\wt\gamma^+_e-\iota\wt\gamma^+_e$ of multiplicity $2$ are the ones passing through the preimage of the weakly-fixed tree. 

We proceed to describe a spanning tree for $\Gamma^{\dagger}$ and the resulting basis of homology. We give $\Gamma^{\dagger}$ the orientation induced from $\Gamma$. 
For each edge $e$ of $\Gamma$ that is not part of the weakly-fixed tree, denote $e^{\dagger}$  the corresponding edge of $\Gamma^{\dagger}$. If $e$ is part of the weakly-fixed tree, denote $e^{\dagger}$ and $e^{\dagger'}$ the two corresponding edges of $\Gamma^{\dagger}$. Now, the tree $\calT^{\dagger}$ will consist of an edge $e^{\dagger}$ for each edge $e$ not in the weakly-fixed tree, the edges $e^{\dagger},e^{\dagger'}$ for each edge $e$ that is part of the weakly-fixed tree, and the edge $e_g^{\dagger}$. The edges of $\Gamma^{\dagger}$ in the complement of $\calT^{\dagger}$ are exactly $e_1^{\dagger},\ldots,e_{g-1}^{\dagger}$. 

Let $e$ be one of those edges and, similarly to before, denote $v^{\dagger}, u^{\dagger}$ and $v_g^{\dagger}, u_g^{\dagger}$ the end points of the edges $e^{\dagger}$ and $e_g^{\dagger}$ respectively. 
If the edge $e$ is adjacent to the weakly-fixed tree then the corresponding cycle starts with a path from $v_g^{\dagger}$ to $v^{\dagger}$ along a copy of the weakly-fixed tree, proceeds with $e^{\dagger}$,  then the unique path along $\calT^{\dagger}$ to $u_g^{\dagger}$, which passes through the other copy of the weakly-fixed tree. 
In other words, for each edge appearing in $\wt\gamma_e^+$, there is a unique corresponding edge appearing in $\gamma_e^{\dagger}$ with compatible orientation. If the edge $e$ is not adjacent to the weakly-fixed tree then the corresponding cycle $\gamma_e^{\dagger}$ does not pass through the weakly-fixed tree and there is an orientation preserving bijection between its edges and the edges of $\gamma_e$, and in particular with the edges of $\wt\gamma_e^{\dagger}$.

\begin{example} 
\cref{cycle} and \cref{abel} show hyperelliptic double covers of metric graphs and their images under the Abel--Prym map. 
In \cref{cycle}, we outline the cycle $\gamma_1$, the anti-symmetric cycle $\widetilde{\gamma}^+_1-\iota\widetilde{\gamma}^+_1$ and its image under the Abel--Prym map $\Psi$. 
 \cref{abel} shows a double cover that has a non-trivial weakly-fixed tree consisting of a single bridge, denoted $b$. Among the elements of the  basis $\calC$, exactly one cycle passes through 
 the bridge.

 Its image under $\Psi$ is highlighted in green.

    \begin{figure}
    \centering
\begin{tikzpicture}
    
\coordinate (11) at (0,1);
\coordinate (12) at (-1,1);
 \coordinate (13) at (-1.5,1);
\coordinate (14) at (-2.5,1); 
\coordinate (11i) at (0,1);
\coordinate (12i) at (-1,1);
\coordinate (13i) at (-1.5,1);
 \coordinate (14i) at (-2.5,1);
 \coordinate (17) at (-0.5,1.5);
 \coordinate (17i) at (-0.5,0.5);
 \coordinate (18) at (-2,1.5);
 \coordinate (18i) at (-2,0.5);
\coordinate (19) at (-1.25,2.2);
 \coordinate (19i) at (-1.25,-1.2+1);

       \coordinate (a11) at (-0.5,6.7);
       \coordinate (a12) at (-1,5);
       \coordinate (a13) at (-1.5,5);
       \coordinate (a14) at (-2.5,5);
        \coordinate (a11i) at (0,5);
       \coordinate (a12i) at (-1,5);
       \coordinate (a13i) at (-1.5,5);
       \coordinate (a14i) at (-2.5,5);
        \coordinate (a17) at (-0.5,5.5);
         \coordinate (a17i) at (-0.5,0.5-1.5+5.5);
       \coordinate (a18) at (-2,5.5);
       \coordinate (a18i) at (-2,0.5-1.5+5.5);
       \coordinate (a19) at (-1.25,1.2+5);
        \coordinate (a19i) at (-1.25,-1.2-0.5+5.5);
       
\coordinate (x) at (0.15,6.7);

       \coordinate (b11) at (0,9-0.6);
       \coordinate (b12) at (-1,9-0.6);
       \coordinate (b13) at (-1.5,9-0.6);
       \coordinate (b14) at (-2.5,9-0.6);
        \coordinate (b11i) at (-0.5,6.7);
       \coordinate (b12i) at (-1,11.5-2-1-0.6+0.5);
       \coordinate (b13i) at (-1.5,4.2+4-0.3+0.5);
       \coordinate (b14i) at (-2.5,-0.5+9-0.6+0.5);
     \coordinate (b17) at (-0.5,9.2-0.3);
      \coordinate (b17i) at (-0.5,0.2-1.5+9-0.3+0.5);
       \coordinate (b18) at (-2,9.2-0.3);
       \coordinate (b18i) at (-2,0.5-1.5+9-0.6+0.5);
       \coordinate (b19) at (-1.25,1.2+9-0.6);
        \coordinate (b19i) at (-1.25,6.7+0.5);

\coordinate (c11) at (5.5,1.3); 
\coordinate (c12) at (3+1.5,1);
 \coordinate (c13) at (-1.5+5.5,1);
\coordinate (c14) at (1.5+1.5,1); 
\coordinate (c11i) at (5.5,0.8);
\coordinate (c12i) at (4.5,1);
\coordinate (c13i) at (-1.5+5.5,1);
 \coordinate (c14i) at (3,1);
 \coordinate (c17) at (-0.5+5.5,1.5);
 \coordinate (c17i) at (-0.5+5.5,0.5);
 \coordinate (c18) at (3.5,1.5);
 \coordinate (c18i) at (3.5,0.5);
\coordinate (c19) at (-1.25+5.5,2.2);
 \coordinate (c19i) at (-1.25+5.5,-0.2);
  
\coordinate (d11) at (5.5,7.25-0.5);
\coordinate (d12) at (4.5,6.5);
 \coordinate (d13) at (4,6.5);
\coordinate (d14) at (3,6.5); 
\coordinate (d11i) at (5.5,6.75-0.5);
\coordinate (d12i) at (4.5,6.5);
\coordinate (d13i) at (4,6.5);
 \coordinate (d14i) at (3,6.5);
 \coordinate (d17) at (5,7);
 \coordinate (d17i) at (5,6);
 \coordinate (d18) at (3.5,7);
 \coordinate (d18i) at (3.5,6);
\coordinate (d19) at (4.25,8.2-0.5);
 \coordinate (d19i) at (4.25,5.8-0.5);
            

 \foreach \i in {d12,d12i,12,12i,17i,17,11,11i,13,13i,14,14i,18,18i,19,19i,a11,x,
 a12,a12i,a13,a13i,a14,a14i,a17,a17i,a18,a18i,a19,a19i,b19i,
 b11i,b12,b12i,b13,b13i,b14,b14i,b17,b17i,b18,b18i,b19,b19i,
 d12,d12i,d13,d13i,d14,d14i,d17,d17i,d18,d18i,d19,d19i}
 \draw[fill=black](\i) circle (0.15em);
 
\draw[blue,fill=blue](d17) circle (0.15em);
\draw[blue,fill=blue](11) circle (0.15em);
\draw[blue,fill=blue](d12) circle (0.15em);
\draw[blue,fill=blue](d17i) circle (0.15em);
\draw[blue,fill=blue](17) circle (0.15em);
\draw[blue,fill=blue](12) circle (0.15em);
\draw[blue,fill=blue](17i) circle (0.15em);

\draw[thin, black] (19)--(18);
\draw[thin, black] (19)--(17);
\draw[thick, black] (18)--(13);
\draw[thick, black] (18)--(14);
\draw[thick, blue] (17)--(12);
\draw[thick, blue] (17)--(11);
\draw[thin, black] (19i)--(18i);
\draw[thin, black] (19i)--(17i);
\draw[thick, black] (18i)--(13i);
\draw[thick, black] (18i)--(14i);
\draw[thick, blue] (17i)--(12i);
\draw[thick, blue] (17i)--(11i);
\draw[thin, black] (14)--(14i);
\draw[thin, black] (13i)--(13);
\draw[thin, red] (12i)--(12);

\draw[thin, black] (a19)--(a18);
\draw[thin, black] (a19)--(a17);
\draw[thick, black] (a18)--(a13);
\draw[thick, black] (a18)--(a14);
\draw[thin, black] (a17)--(a12);
\draw[thin, black] (a17)--(a11);
\draw[dashed, black] (a19i)--(a18i);
\draw[thin, black] (a19i)--(a17i);
\draw[thick, black] (a18i)--(a13i);
\draw[thick, black] (a18i)--(a14i);
\draw[thin, black] (a17i)--(a12i);
\draw[thin, red] (a13i)--(a13);
\draw[thin, red] (a12i)--(a12);

\draw[dashed, black] (b19)--(b18);
\draw[thin, black] (b19)--(b17);
\draw[thick, black] (b18)--(b13);
\draw[thick, black] (b18)--(b14);
\draw[thin, black] (b17)--(b12);
\draw[thin, black] (b19i)--(b18i);
\draw[thin, black] (b19i)--(b17i);
\draw[thick, black] (b18i)--(b13i);
\draw[thick, black] (b18i)--(b14i);
\draw[thin, black] (b17i)--(b12i);
\draw[thin, black] (b17i)--(b11i);
\draw[thin, red] (a11)--(b11i);
\draw[thin, red] (b13i)--(b13);
\draw[thin, red] (b12i)--(b12);
\draw[thin, red] (b14)--(b14i);
\draw[thin, red] (a14)--(a14i);
 \draw[thin, black] (b17) to [out=-60, in=60] (a17i);

\draw[dashed, black] (d19)--(d18);
\draw[thin, black] (d19)--(d17);
\draw[thick, black] (d18)--(d13);
\draw[thick, black] (d18)--(d14);
\draw[thick, blue] (d17)--(d12);

\draw[thin, black] (d19i)--(d18i);
\draw[thin, black] (d19i)--(d17i);
\draw[thick, black] (d18i)--(d13i);
\draw[thick, black] (d18i)--(d14i);
\draw[thick, blue] (d17i)--(d12i);


\draw[thin, red] (c14)--(c14i);
\draw[thin, red] (c13i)--(c13);
\draw[thin, red] (c12i)--(c12);

\draw (-2.5,1.8) node [below] {$\gamma_1$};
\draw (0,1.8) node [below] {$\gamma_g$};
\draw (-2.5,1.8+4) node [below] {$\iota\widetilde{\gamma}^+_1$};
\draw (-2.5,1.8+7.5) node [below] {$\widetilde{\gamma}^+_1$};
  \draw (-1.2,3.2) node [below] {$\downarrow$};
    \draw (-0.9,3) node [below] {$\pi_F$};
 \draw (1.4,6.8) node [below] {$\rightarrow$};
 \draw (1.4,6.6) node [above] {$\psi$};
  \draw (2.6,6.6) node [above] {$\psi(\widetilde{\gamma}^+_1)$};
\draw (0.1,1) node [below] {$x$};
  \end{tikzpicture}
 \caption{The double cover of a graph  associated with $F=\{x\}$ and the corresponding Abel--Prym graph. 
  The dashed edges are an instance of edges identified by the map $\Psi$. The blue cycle is an example of a cycle that gets broken by $\Psi$. 
  }

    \label{cycle}
\end{figure}
\end{example}

\begin{figure}[h]
    \centering
\input{pic4}
 \caption{The Abel--Prym map of a double cover of a graph of genus $5$ with  weakly-fixed subtree $T_f=b$. We highlight with a dashed line an instance of two edges identified by $\Psi$. 
 In green, the cycle $\gamma_e^{\dagger}\in H_1(\Ga^{\dagger})$  coming from an edge $e$ adjacent to the weakly-fixed tree.  The squiggly arrows show the steps of constructing $\Gamma^{\dagger}$ from $\Gamma$. 
}
    \label{abel}
\end{figure}

Using the bases of cycles constructed above, we now relate the Prym variety of a double cover with the Jacobian of the Abel--Prym graph.

\begin{theorem}
\label{pptavs}
    The Prym variety of a hyperelliptic double cover is isomorphic, as pptav, to the Jacobian of the corresponding Abel--Prym graph, namely
    \[
    \Prym(\tGa/\Ga)\cong \mathrm{Jac}(\Psi(\tGa)).
    \]
    \begin{proof}

        We will show that there exists an isomorphism of integral tori $f=(f^{\#},f_{\#})$ between the Prym variety and the Jacobian of the Abel--Prym graph. Let $\calT^{\dagger}$ and $\wt\calT$ be the spanning trees for $\Gamma^{\dagger}$ and $\wt\Gamma$ described above and let $\calB^{\dagger}$ and $\wt\calC$ be the corresponding bases for homology and antisymmetric cycles. 
       We  maintain the  terminology introduced during the construction of the bases. 
         Define
\[
 \begin{tabular}{cccc}
	    $f_{\#} \colon$&$(\mathrm{Ker}\,\pi_*\colon H_1(\tGa)\to H_1(\Ga))$&$\longrightarrow$&$H_1(\Psi(\wt\Gamma))$  \\
	         &$\tga_{i}^+-\iota \tga_{i}^+$&$\mapsto$&$\Psi_*(\tga_{i}^+)$
	    \end{tabular}.
\]
The map $f^\#$ is defined  in a similar way:
under the identification of $H_1(\Psi(\tGa),\mathbb Z)$  with $H^1(\Psi(\tGa),\mathbb Z)$ using the principal polarisation, we define $f^{\#}$ as the map sending a cycle $\gamma_i^{\dagger}$ in $\calB^{\dagger}$  to the class of $\wt\gamma_i^+ - \iota\wt\gamma_i^+$ in the cokernel. 

It is clear from the description of the cycles that the various images $\Psi(\wt\gamma_i^+)$ are exactly the cycles constructed from the edges in the complement of the spanning tree $\calT^{\dagger}$, so  the maps $f_{\#}$ and $f^{\#}$ are isomorphisms.

We are left with proving that the integration pairing is compatible on both sides.
It suffices to verify this property on elements of the bases, namely that
\[
[f_{\#}(\wt\gamma_i^{+} - \iota(\wt\gamma_i^{+}), \gamma_j^{\dagger}]_1 =
[\wt\gamma_i^{+} - \iota(\wt\gamma_i^{+}), f^{\#}(\gamma_j^{\dagger})]_2 
\]
for all $i,j\in\{1,\ldots,g-1\}$. From the definitions of $f_{\#}$ and $f^{\#}$, this is equivalent to checking that
\[
[\gamma_i^{\dagger}, \gamma_j^{\dagger}]_1 =
[\wt\gamma_i^{+} - \iota(\wt\gamma_i^{+}), \wt\gamma_j^{+} - \iota(\wt\gamma_j^{+})]_2
\]
for all $i,j$.

From the description of the cycles, there is an orientation preserving bijection between the edges of $\wt\gamma_i^+$ and $\gamma_i^{\dagger}$. It what follows, we will analyse the contribution coming from each edge $e$ of $\Gamma$ to the two bilinear forms. 
Let $e$ be an edge of $\Gamma$, and denote by $\epsilon_i,\epsilon_j\in\{-1,0,1\}$ the coefficients in which 
 $e^{\dagger}$  appears in $\gamma_i^{\dagger}$ and $\gamma_j^{\dagger}$    (with respect to the fixed orientation). If $e$ doesn't belong to the weakly-fixed tree then the length of $e^{\dagger}$ is $\ell(e)$, so  it contributes $\epsilon_i\epsilon_j\ell(e)$ to $[\gamma_i^{\dagger},\gamma_j^{\dagger}]$. At the same time, there is a corresponding edge $\tilde e$ appearing in $\wt\gamma_i^{+}$ and $\wt\gamma_j^{+}$ with the same coefficients $\epsilon_i,\epsilon_j$ respectively.  As noted before, since $e$ is not part of the weakly-fixed tree,  the edge $\iota e$ does not appear in $\iota\wt\gamma^+$, so $e$ contributes  $2\epsilon_i\epsilon_j\ell(e)$ to the intersection pairing between $\wt\gamma_i^{+} - \iota\wt\gamma_i^{+}$ and $\wt\gamma_j^{+} - \iota\wt\gamma_j^{+}$. Since the bilinear form $[,]_2$ is obtained from the intersection pairing by division by 2, the edge contributes $\epsilon_i\epsilon_j\ell(e)$ to $[\wt\gamma_i^{+} - \iota\wt\gamma_i^{+}, \wt\gamma_j^{+} - \iota(\wt\gamma_j^{+})]_2$, as anticipated.

If, on the other hand, the edge $e$ does come from the weakly-fixed tree, the length of $e^{\dagger}$  is $2\ell(e)$ so it contributes $\epsilon_i,\epsilon_j\cdot 2\ell(e)$ to $[\gamma_i^{\dagger},\gamma_j^{\dagger}]$. The length of  the corresponding edge $\tilde e$  in $\wt\Gamma$ is $\ell(e)$. 
In this case, however, the edge $\iota\tilde e$ does appear in $\wt\gamma^{+}_i$, so it appears with multiplicities $2\epsilon_i,2\epsilon_j$ in $\wt\gamma_i^{+}$ and $\wt\gamma_j^{+}$ respectively. The edge therefore contributes  $4\epsilon_i\epsilon_j\ell(e)$ to the intersection pairing, therefore $2\epsilon_i\epsilon_j\ell(e)$ to the bilinear form. Either way, each edge of the graph contributes the same amount to both bilinear forms, so they coincide.

    \end{proof}
\end{theorem}

 \subsection{The tropical bigonal construction and the Abel--Prym graph}

In this final section, we explore a strong connection, illuminated by the referee, between the image of the Abel--Prym map and the so-named bigonal construction.
We briefly recall the bigonal construction now and  refer to \cite{RZ_ngonal} for additional details and to \cite{Zakharov_Trigonal, RohrleSaillez_Donagi} for more recent developments on the topic.

Let $\pi\colon \wt\Gamma\to \Gamma$ be a harmonic double cover and $f\colon \Gamma \to K$
a harmonic morphism of degree 2, where $K$ is a metric tree. While the bigonal construction allows $\pi$ to be any harmonic morphism,  we will assume that it is free as we don't require the more general case. 
The output of the construction
 is a certain degree $4$ harmonic morphism from a graph $\wt\Pi$ to the tree $K$.

As a set, $\wt \Pi$  is
given by
\[
\wt{\Pi}=\{x_1+x_2\in \mathrm{Div}^2_+(\tGa)\big|\,\exists\, x\in K : \pi(x_1)+\pi(x_2)=\sum_{y\in f^{-1}(x)}d_k(y)\cdot y\}.
\]
Namely, for each point $x$ of $K$, we consider its pullback $f^*(x)$ to $\mathrm{Div}(\Gamma)$ (which consists of either two points of multiplicity 1 or a single point of multiplicity 2), then take all the degree 2 divisors in $\wt\Gamma$ whose pushforward to $\mathrm{Div}(\Gamma)$ is $f^*(x)$.

The morphism $\tilde p:\wt \Pi\to K$ is then given by sending to $x\in K$ all the points of $\wt \Pi$ induced by it. 
From \cite[Proposition 2.1]{RZ_ngonal}, the map $\tilde p$ is harmonic of degree $4$. When $f^*(x)$ consists of two distinct point, the fibre $\tilde p^{-1}(x)$ will consist of 4 distinct points of $\wt \Pi$, as there are two choices for the preimage of each point in the support of $f^*(x)$.  
When  $f^*(x)$ consists of a single point $y$ with multiplicity 2, the fibre $\tilde p^{-1}(x)$ will consist of $3$ distinct points:  denoting by $z$ a point of $\wt\Gamma$ such that $\pi(z) = y$,  the corresponding points of $\wt \Pi$ are $z+\tj z, 2z,$ and $2\tj z$. In the language of \cite{RZ_ngonal}, these points are referred to as Type IV and Type II respectively. 

When the fibre over an edge $e$ of $K$ consists of 4 different edges of $\wt \Pi$, each of those edges will be given length equal to the length of $e$ and $\tilde p$ is an isometry on those edges. When the fibre consists of 3 edges, the edges containing points of the form $z+\tj z$ will be  given length equal to the length of $e$ and $\tilde p$ maps them isometrically to $K$. 
Edges containing points of the form 
$2z$ and $2\tj z$ will be given  length equal to half the length of $e$ and $\tilde p$ will dilate those edges.

\begin{example}
   \cref{fig:counterexamplebigonal2} depicts the bigonal construction for a double cover of three loops.  \cref{fig:counterexamplebigonal} shows the Abel--Prym map for the same double cover. The image of the Abel--Prym map is almost identical to a component of the output of the bigonal construction. This will be explained in the next result.

    \begin{figure}[H]
    \centering
\begin{tikzpicture}
	\coordinate (0x) at (-5+1.5,0);
    \coordinate (5x) at (1-1.5,0);
	\coordinate (1x) at (0-1.5,0);
	\coordinate (3x) at (-4+1.5,0);
	\coordinate (a1x) at (1.5-4+2.5-1.5,2);
	\coordinate (a3x) at (-4+1.5,2);
	\coordinate (b1x) at (1.5-4+2.5-1.5,5);
	\coordinate (b3x) at (-4+1.5,5);
	\coordinate (2x) at (-2.5,0);
	\coordinate (4x) at (-1.5,0);
        \coordinate (a2x) at (-2.5,2);
	\coordinate (a4x) at (-1.5,2);
	\coordinate (b2x) at (-2.5,5);
	\coordinate (b4x) at (-1.5,5);
 \coordinate (a5x) at (1-1.5,2);
  \coordinate (b5x) at (1-1.5,5);
  \coordinate (a0x) at (-4.6+1.85,3.5);
  \coordinate (b0x) at (-3.4+1.15,3.5);

 \coordinate (0tree) at (-5+1.5,-2);
  \coordinate (5tree) at (1-1.5,-2);
                \coordinate (1tree) at (1.5-2+0.5-1.5,-2);
	\coordinate (3tree) at (-4+1.5,-2);
    \coordinate (2tree) at (-2.5,-2);
	\coordinate (4tree) at (-1.5,-2);

    \coordinate (0treer) at (-5+9+1.5,-2);
  \coordinate (5treer) at (10-1.5,-2);
                \coordinate (1treer) at (1.5-2+0.5+9-1.5,-2);
	\coordinate (3treer) at (-4+9+1.5,-2);
    \coordinate (2treer) at (-2.5+9,-2);
	\coordinate (4treer) at (-1.5+9,-2);

     \coordinate (a0treer) at (-5+9+1.5,2.2-2);
  \coordinate (a5treer) at (10-1.5,2.2-2);
                \coordinate (a1treer) at (1.5-2+0.5+9-1.5,2.2-2);
	\coordinate (a3treer) at (-4+9+1.5,0.2);
    \coordinate (a2treer) at (-2.5+9,2.2-2);
	\coordinate (a4treer) at (-1.5+9,2.2-2);

 \coordinate (a'0treer) at (-5+9+1.5,1.8-2);
  \coordinate (a'5treer) at (10-1.5,1.8-2);
                \coordinate (a'1treer) at (1.5-2+0.5+9-1.5,1.8-2);
	\coordinate (a'3treer) at (-4+9+1.5,1.8-2);
    \coordinate (a'2treer) at (-2.5+9,1.8-2);
	\coordinate (a'4treer) at (-1.5+9,1.8-2);

     \coordinate (b0treer) at (-5+9+1.5,2.8);
  \coordinate (b5treer) at (10-1.5,3.3);
                \coordinate (b1treer) at (1.5-2+0.5+9-1.5,3.3);
	\coordinate (b3treer) at (-4+9+1.5,3.3);
    \coordinate (b2treer) at (-2.5+9,3.3);
	\coordinate (b4treer) at (-1.5+9,3.3);

 \coordinate (b'0treer) at (-5+9+1.5,2.8);
  \coordinate (b'5treer) at (10-1.5,2.3);
                \coordinate (b'1treer) at (9-1.5,2.3);
	\coordinate (b'3treer) at (-4+9+1.5,2.3);
    \coordinate (b'2treer) at (-2.5+9,2.3);
	\coordinate (b'4treer) at (-1.5+9,2.3);

     \coordinate (d1) at (9-1.5,5);
	\coordinate (d3) at (5+1.5,5);
    \coordinate (d2) at (-2.5+9,5);
	\coordinate (d4) at (-1.5+9,5);
\coordinate (d0) at (4+1.5,5.5);
\coordinate (d0') at (4+1.5,4.5);
    \coordinate (d5) at (10-1.5,5);
		\foreach \i in {d0,d0',d1,d2,d3,d4, 1tree,2tree,3tree,5tree,4tree,0tree, 1treer,2treer,3treer,5treer,4treer,0treer, a1treer,a2treer,a5treer,a4treer,a0treer, a'1treer,a'2treer,a'3treer,a3treer,
        a'5treer,a'4treer,a'0treer,  b1treer,b2treer,b3treer,b5treer,b4treer,b0treer, b'1treer,b'2treer,b'3treer,b'5treer,b'4treer}
		\draw[fill=black](\i) circle (0.15em);
		\draw[thick, black] (3x)--(2x);
		\draw[thick, black] (4x)--(1x);

        \foreach \i in {b'0treer,a1x,a2x,a4x,b0treer,b2x,b4x,a3x,b1x,b3x,0x,5x,a5x,b5x,1x,a0x,b0x,2x,4x,3x, d5, d3, d2, d4, d1}
		\draw[fill=black](\i) circle (0.30em);

\node[circle, draw, minimum size=1cm] (c) at (9.5-1.5,5) {};
\node[circle, draw, minimum size=1cm] (c) at (7,5) {};
            
				\node[circle, draw, minimum size=1cm] (c) at (-0.5-4+1.5,0) {};
                \node[circle, draw, minimum size=1cm] (c) at (-2,2) {};
                \node[circle, draw, minimum size=1cm] (c) at (-2,5) {};
				\node[circle, draw, minimum size=1cm] (c1) at (2-4+1,0) {};
					\node[circle, draw, minimum size=1cm] (c3) at (-2,0) {};
				\node[circle, draw, minimum size=1cm] (c3) at (0.5-1.5,2) {};
                \node[circle, draw, minimum size=1cm] (c3) at (0.5-1.5,5) {};
				\draw[thick, black] (a3x)--(a2x);
                \draw[thick, black] (a4x)--(a1x);
				\draw[thick, black] (b3x)--(b2x);
                \draw[thick, black] (b4x)--(b1x);
			 \draw[thin, black] (-5+1.5,-2)--(1-1.5,-2);
             \draw[thick, black] (d0')--(d3);
             \draw[thick, black] (d0)--(d3);
				 \draw[thin, black] (b3x) to [out=180+75, in=90+15] (a3x);
 \draw[thin, black] (b3x) to [out=-75, in=75] (a3x);
 
\draw[thin, black] (4+1.5,-2)--(10-1.5,-2);
            \draw[thin, black] (4+1.5,-0.2)--(10-1.5,-0.2);

            \draw[thin, black] (4+1.5,0.2)--(10-1.5,0.2);
\draw[thick, black] (d3)--(d2);
\draw[thick, black] (b0treer)--(b3treer);
\draw[thick, black] (b'0treer)--(b'3treer);
\draw[thick, black] (d4)--(d1);
 \draw[thin, black] (5+1.5,3.3)--(10-1.5,3.3);
  \draw[thin, black] (5+1.5,2.3)--(10-1.5,2.3);
				\draw (-2,1.4) node [below] {$\downarrow$};
                \draw (7,1.4) node [below] {$\downarrow$};
                \draw (-2,-0.8) node [below] {$\downarrow$};
                \draw (7,-0.8) node [below] {$\downarrow$};
				\draw (-5.5,4) node [above] {$\tGa$};
                \draw (-5.5,4) node [below] {$\xdownarrow{1.6cm}$};
                \draw (-5.5,0) node [above] {$\Ga$};
                \draw (-5.5,0) node [below] {$\Bigg\downarrow$};
                 \draw (-5.5,-2) node [above] {$K$};

                 \draw (3.5,3.7) node [above] {$\widetilde \Pi$};
                 \draw (3.5,3.7) node [below] {$\xdownarrow{1.5cm}$};
                \draw (3.5,0) node [above] {$\Pi$};
                 \draw (3.5,0) node [below] {$\Bigg\downarrow$};
                 \draw (3.5,-2) node [above] {$K$};
			
					\end{tikzpicture}
 \caption{The bigonal construction for the double cover in \cref{fig:counterexamplebigonal}. The thickness on the vertices corresponds to the dilation factor given by the cover maps, namely the fixed points of the involutions.}
    \label{fig:counterexamplebigonal2}
\end{figure}

     \begin{figure}[H]
    \centering
\begin{tikzpicture}

	\coordinate (1) at (6,3.5);
	\coordinate (1x) at (1.5-2+0.5-1.5,0);
	\coordinate (3x) at (-4+1.5,0);
	\coordinate (a1x) at (1.5-4+2.5-1.5,2);
	\coordinate (a3x) at (-4+1.5,2);
	\coordinate (b1x) at (1.5-4+2.5-1.5,5);
	\coordinate (b3x) at 
    (-4+1.5,5);
	\coordinate (2x) at (-2.5,0);
	\coordinate (4x) at (-1.5,0);
        \coordinate (a2x) at (-2.5,2);
	\coordinate (a4x) at (-1.5,2);
	\coordinate (b2x) at (-2.5,5);
	\coordinate (b4x) at (-1.5,5);
				
		\foreach \i in {1x,2x,4x,3x,a1x,1,a2x,a4x,b2x,b4x,a3x,b1x,b3x}
		\draw[fill=black](\i) circle (0.15em);
		\draw[thin, black] (3x)--(2x);
		\draw[thin, black] (4x)--(1x);

\node[circle, draw, minimum size=1cm] (c) at (5.5,3.5) {};
\node[circle, draw, minimum size=1cm] (c) at (6.5,3.5) {};

				\node[circle, draw, minimum size=1cm] (c) at (-0.5-4+1.5,0) {};
                \node[circle, draw, minimum size=1cm] (c) at (-2,2) {};
                \node[circle, draw, minimum size=1cm] (c) at (-2,5) {};
				\node[circle, draw, minimum size=1cm] (c1) at (2-4+2.5-1.5,0) {};
					\node[circle, draw, minimum size=1cm] (c3) at (-2,0) {};
				\node[circle, draw, minimum size=1cm] (c3) at (0.5-1.5,2) {};
                \node[circle, draw, minimum size=1cm] (c3) at (0.5-1.5,5) {};
				\draw[thin, black] (a3x)--(a2x);
                \draw[thin, black] (a4x)--(a1x);
				\draw[thin, black] (b3x)--(b2x);
                \draw[thin, black] (b4x)--(b1x);
			
				 \draw[thin, black] (b3x) to [out=180+75, in=90+15] (a3x);
 \draw[thin, black] (b3x) to [out=-75, in=75] (a3x);

            \draw (-4,3.5) node [above] {$\tGa$};
                \draw (-4,0.5) node [above] {$\Ga$};
				\draw (-2,1.4) node [below] {$\downarrow$};
                \draw (-2,1.1) node [right] {$\pi$};
				\draw (3,3.5) node [below] {$\rightarrow$};
				\draw (3,3.5) node [above] {$\Psi$};
			
					\end{tikzpicture}
 \caption{Free double cover of the chain of $3$ loops by a chain of $5$ loops and the image $\Psi(\tGa)$ on the right.}
    \label{fig:counterexamplebigonal}
\end{figure}

\end{example}

The following result relates the bigonal construction with the Abel--Prym map.

 \begin{theorem}\label{thm:bigonalconstruction}
     Let $\pi\colon\tGa\to\Ga$ be a free double cover, where $\tGa$ and $\Ga$ are hyperelliptic metric graphs with involutions  $\tj$ and $j$ respectively,  such that the Abel--Prym map is finite. 
     Then the output of the tropical bigonal construction consists of two connected components. 
     One is isomorphic to the Abel--Prym graph and the other to the  tree $\tGa/\tj$, where the quotient is taken with respect to the minimal model for which $\tj$ doesn't flip any edges.

  \end{theorem}
  
 \begin{proof}
 Let $x\in K$. Then $f^*(x) = y + j(y)$ for some $y\in\Gamma$. If $z\in\wt\Gamma$ is such that $\pi(z) = y$, then the points of $\wt\Pi$ in the fibre of $x$ are the divisors
 \[
 z + \tj z,\,\,\, z+\iota\tj z,\,\,\, \iota z + \tj z,\,\,\, \text{ and } \iota z + \iota\tj z.
 \]

We will show that one connected component of $\wt \Pi$, denoted $\Pi_1$  consists of the points of the form $w + \iota\tj w$ and is isomorphic to $\wt\Gamma/(\iota\circ\tj)$, while the other, denoted $\Pi_2$,  consists of  points of the form $w + \tj w$ and is isomorphic to $\wt\Gamma/\tj$.

Define a map  $\psi:\tGa/(\iota\circ\tj)\to \Pi_1$ by sending each class $[z]\in \tGa/(\iota\circ\tj)$ to $z + \iota\tj z$. The map is well defined on equivalence classes since it sends $z$ and $\iota\tj z$ to the same point in $\Pi$. It is injective since, whenever $\psi(z) = \psi(w)$, it follows that $z + \iota\tj z = w + \iota\tj w$, so $\Psi(z) = \Psi(w)$. It is surjective onto $\Pi_1$ since, for all $w+\iota\tj w$, we have $\psi(w) = w+\iota\tj w$. The map is also an isometry: by definition, the edges of $\Pi$ that are dilated by a factor of 2 are precisely those mapping to bridges of $K$. By the assumption that $\wt\Gamma$ is 2-edge-connected, the preimage in $\wt\Gamma$ of the bridges of $K$ are precisely the properly disconnected pairs, namely the edges dilated by a factor of 2 by $\Psi$ as in \cref{thm:behaviousOnEdges}. 

Next, we define a map $\phi:\wt\Gamma/\tj\to \Pi_2$ via $\phi([z]) = z + \tj z$. Then $\phi$ is well-defined since it maps $z$ and $\tj z$ to the same point of $\Pi_2$ and is an isomorphism similarly to $\psi$.

Now, the images of $\psi$ and $\phi$ are disjoint and are each connected. From \cite[Proposition 2.5]{RZ_ngonal}, the graph $\wt \Pi$ has two connected components. 
Since the union of the images is all of $\wt \Pi$, it follows that each image is a connected component of $\wt \Pi$.

 \end{proof}

\begin{remark}\label{rem:bigonalForNonFiniteAP}
When the Abel--Prym map is not finite, $\Psi(\wt\Gamma)$ differs from a component of the bigonal construction in two ways. First,  the bridges of   $\wt\Gamma$ are contracted by $\Psi$ but are left intact by the bigonal construction.  Second, if the cyclic $\Psi$-collapsible locus is non-empty, then from \cref{lem:CollapsibleLocusOfHyperelliptic} it consists of a single properly disconnecting pair covering a non-disconnecting edge. The edges of the pair are  folded in half and become leaves in the bigonal construction. Said differently, in the 2-edge-connected (but not necessarily finite) case, the bigonal construction gives rise to the topological quotient $\wt\Gamma/(\iota\circ\tj)$ whereas the Abel--Prym graph is the quotient with respect to the canonical model. 
\end{remark}

Note that the hyperellipticity assumption on $\wt\Gamma$ is necessary in the theorem. Namely, it's not true in general that the Abel--Prym graph coincides with a component of the bigonal construction.

\begin{example}\label{ex:bigonalForNonHyperelliptic}
    Let $\Gamma$ be the hyperelliptic chain of 3 loops and $\wt\Gamma$ the 2-edge-connected non-hyperelliptic free double cover as in \cref{fig:nonhypcase}.  The Abel--Prym graph has genus 3. On the other hand, the bigonal construction gives rise to two connected components of genus 1, see Figure 2 in \cite[Section 2.2]{RZ_ngonal}.

            \begin{figure}[h]
                \centering\begin{tikzpicture}
				
				\coordinate (1x) at (0,0);
				\coordinate (3x) at (-4,0);
				\coordinate (a1x) at (0,2);
				\coordinate (a3x) at (-4,2);
				\coordinate (b1x) at (1.5-4+2.5,5);
				\coordinate (b3x) at (-4,5);

                \coordinate (da1x) at (1.5+4+2.5+1,2);
				\coordinate (da3x) at (1+3-1,2);
				\coordinate (db1x) at (1.5-4+2.5+5+3+1,5);
				\coordinate (db3x) at (1+3-1,5);
\coordinate (da2x) at (-2.5+5+3,2);
					\coordinate (da4x) at (-1.5+5+3,2);
						\coordinate (db2x) at (-2.5+5+3,5);
					\coordinate (db4x) at (-1.5+5+3,5);

					\coordinate (2x) at (-2.5,0);
					\coordinate (4x) at (-1.5,0);
						\coordinate (a2x) at (-2.5,2);
					\coordinate (a4x) at (-1.5,2);
						\coordinate (b2x) at (-2.5,5);
					\coordinate (b4x) at (-1.5,5);
				
				\foreach \i in {1x,2x,4x,3x,da1x,a1x,da2x,a2x,da4x,a4x,db2x,b2x,db4x,b4x,da3x,a3x,db1x,b1x,db3x,b3x}
				\draw[fill=black](\i) circle (0.15em);
				\draw[thin, black] (3x)--(2x);
					\draw[thin, black] (4x)--(1x);

				\node[circle, draw, minimum size=1cm] (c) at (-0.5-4,0) {};
				\node[circle, draw, minimum size=1cm] (c1) at (2-4+2.5,0) {};
                \node[circle, draw, minimum size=1cm] (c1l) at (-2,2) {};
                \node[circle, draw, minimum size=1cm] (c1ll) at (-2,5) {};
                \node[circle, draw, minimum size=1cm] (dc1l) at (-2+5+3,2) {};
                \node[circle, draw, minimum size=1cm] (dc1ll) at (-2+5+3,5) {};
					\node[circle, draw, minimum size=1cm] (c3) at (-2,0) {};
				
				\draw[thin, black] (a3x)--(a2x);
                \draw[thin, black] (a4x)--(a1x);
				\draw[thin, black] (b3x)--(b2x);
			\draw[thin, black] (b4x)--(b1x);
            \draw[thin, black] (da3x)--(da2x);
                \draw[thin, black] (da4x)--(da1x);
				\draw[thin, black] (db3x)--(db2x);
			\draw[thin, black] (db4x)--(db1x);
				\draw[thin, black] (da1x)--(db1x);
			\draw[thin, black] (db3x)--(da3x);
				\draw (-2,1.3) node [below] {$\downarrow$};
				\draw (2,3.5) node [below] {$\rightarrow$};
\draw (2,3.5) node [above] {$\Psi$};
                
				\draw[thin, black] (b3x) to [out=-45, in=45] (a3x);
				\draw[thin, black] (b3x) to [out=180+45, in=90+45] (a3x);
				\draw[thin, black] (b1x) to [out=-45, in=45] (a1x);
				\draw[thin, black] (b1x) to [out=180+45, in=90+45] (a1x);
			
					\end{tikzpicture}
                \caption{A free double cover where the source is not hyperelliptic and the corresponding Abel--Prym graph.}
          \label{fig:nonhypcase}
            \end{figure}
				
\end{example}

\begin{question}
Does the Abel--Prym graph coincide with a component of the bigonal construction only in the hyperelliptic case?  Is there a natural connection between the two constructions when they don't coincide?      
\end{question}

The involution $\iota$ on $\wt\Gamma$ induces an involution on  $\wt \Pi$  that we also denote $\iota$ by abuse of notation. Modding by $\iota$, we see that $\wt \Pi$ is a (ramified) harmonic double cover of a graph $\Pi$. Furthermore, $\Pi$ consists of two trees. 
It is straightforward that the map $\wt \Pi\to K$ factors through the  double cover $\wt \Pi\to \Pi$. In particular, we get a tower $\wt \Pi\to \Pi\to K$ of double covers. 
As a consequence, we obtain the following theorem extending Theorem 5.6 in \cite{RZ_ngonal} to the free case.

\begin{corollary}\label{cor:bigonalConstruction}

    Let $\tGa\overset{\pi}{\to}\Ga\overset{f}{\to}K$ be a tower of harmonic double covers  of metric graphs, where $\pi$ is free and $\wt\Gamma$ is 2-edge-connected, and let $\wt \Pi$ and $\Pi$ be the graphs obtained from the bigonal construction. 
    Then there is an isomorphism of polarised tropical abelian varieties
    \[
\Prym(\tilde{\Pi}/\Pi)^{\vee}\cong \Prym(\tGa/\Ga).
\]
Furthermore, the bigonal construction is an involution in this case. That is, the double cover obtained by applying the bigonal construction to $\wt \Pi/\Pi$ is $\wt\Gamma/\Gamma$.

\end{corollary}

\begin{proof}
    We may assume that the Abel--Prym map is finite since contracting the $\Psi$-contracting locus does not affect the Prym variety. 
    The fact that the construction is an involution follows from \cite[Proposition 2.4]{RZ_ngonal}, since the double cover $\wt\Gamma/\Gamma$ is generic in the sense of \cite[Definition 2.3]{RZ_ngonal}. 
From \cref{mainCorollary:JacobianOfImage}, the Prym variety $\Prym(\wt\Gamma/\Gamma)$ is isomorphic to the Jacobian of $\Psi(\tGa)$.  As shown in \cref{thm:bigonalconstruction}, the graph $\wt \Pi$ is the disjoint union of a tree and $\Psi(\tGa)$, equipped with a ramified double cover of a disjoint union of trees.  The Prym variety of a graph covering a tree coincides with the Jacobian of the graph. The result now follows.

\end{proof}

  \bibliographystyle{alpha}
 \bibliography{PrymBib}

 \end{document}